\documentclass[12pt]{amsart}
\usepackage{fullpage}
\usepackage{amsmath}
\usepackage{amssymb}
\usepackage[none]{hyphenat}
\usepackage{tikz}
\usepackage{float}

\newcommand{\abs}[1]{\left\vert#1\right\vert}
\newcommand{\set}[1]{\left\{#1\right\}}

\newcommand{\ds}{\displaystyle}
\newcommand{\p}[1]{\left(#1\right)}

\newtheorem{thm}{Theorem}[section]
\newtheorem{cor}[thm]{Corollary}
\newtheorem{lem}[thm]{Lemma}
\newtheorem{prop}[thm]{Proposition}
\newtheorem{defn}[thm]{Definition}

\begin{document}
\title{Identifying Complex Hadamard Submatrices of the Fourier Matrices via Primitive Sets}
\author{John E. Herr \& Troy M. Wiegand}
\address{Department of Mathematics, Statistics, \& Actuarial Science, Butler University, Jordan Hall 270, 4600 Sunset Ave., Indianapolis, IN 46208}
\email{jeherr@butler.edu, tmwiegan@butler.edu}
\date{\today}

\begin{abstract}For a given selection of rows and columns from a Fourier matrix, we give a number of tests for whether the resulting submatrix is Hadamard based on the primitive sets of those rows and columns. In particular, we demonstrate that whether a given selection of rows and columns of a Fourier matrix forms a Hadamard submatrix is exactly determined by whether the primitive sets of those rows and columns are compatible with respect to the size of the Fourier matrix. This allows the partitioning of all submatrices into equivalence classes that will consist entirely of Hadamard or entirely of non-Hadamard submatrices and motivates the creation of compatibility graphs that represent this structure. We conclude with some results that facilitate the construction of these graphs for submatrix sizes 2 and 3.
\end{abstract}

\maketitle

\section{Introduction}

An $m\times m$ matrix $H$ is said to be a complex Hadamard matrix if its entries are complex numbers of unit modulus and its columns and rows are orthogonal. That is, $H$ is complex Hadamard if it is of the form $$H=\begin{bmatrix}e^{2\pi i\lambda_{11}} & e^{2\pi i\lambda_{12}} &\cdots & e^{2\pi i \lambda_{1m}}\\e^{2\pi i\lambda_{21}} & e^{2\pi i\lambda_{22}} & \cdots & e^{2\pi i\lambda_{2m}}\\\vdots &\ddots & & \vdots\\\vdots & & \ddots & \vdots \\e^{2\pi i\lambda_{m1}}&e^{2\pi i\lambda_{m2}}&\cdots&e^{2\pi i\lambda_{mm}}\end{bmatrix}$$ where $\lambda_{jk}\in \mathbb{R}$ and $$H^\ast H=H H^\ast = m\cdot I_{m},$$ where $I_m$ is the $m\times m$ identity matrix. We will omit the adjective ``complex'' throughout the rest of the paper. Two Hadamard matrices $H_1$ and $H_2$ are said to be $\textit{equivalent}$ if there exist permutation matrices $P_1$ and $P_2$ and unitary diagonal matrices $D_1$ and $D_2$ such that $H_1=P_1D_1H_2D_2P_2$.

Older research on Hadamard matrices focused mainly on the real-valued case, where all entries are $-1$ or $1$, but in recent times more attention has been placed on the general complex-valued case. Our particular motivation comes from applications in harmonic analysis, in which of special importance are Hadamard matrices arising as submatrices of the Fourier matrices.

The Fourier matrix $\mathcal{F}_m$ is the $m\times m$ matrix whose $(j,k)$th entry is $$\p{\mathcal{F}_m}_{jk}=\p{e^{2\pi i(j-1)(k-1)/m}}_{jk}.$$ In the above definition, $j$ and $k$ run from $1$ to $m$. However, it is often more convenient to think of the rows and columns of a Fourier matrix as being indexed from $0$ to $m-1$, so that
$$\p{\mathcal{F}_{m}}_{jk}=\p{e^{2\pi ijk/m}}_{jk}$$
for $j$ and $k$ running from $0$ to $m-1$. When we select rows and columns from $\mathcal{F}_m$ in this paper, will we regard it as being indexed in this latter manner.

\begin{defn}
    Let $J,K\subseteq\set{0,1,\ldots,m-1}$, with $\abs{J}=\abs{K}=n\leq m$, be a selection of rows and columns of $\mathcal{F}_m$. By $H_{J,K,m}$ we refer to the submatrix of $\mathcal{F}_m$ formed by keeping the rows $J$ and the columns $K$. That is, if $(j_1,j_2,\ldots,j_n)$ is a list of elements of $J$ ordered from least to greatest, and $(k_1,k_2,\ldots,k_n)$ is a list of elements of $K$ ordered from least to greatest, then the $(a,b)$th entry of $H_{J,K,m}$ is
    
    $$(H_{J,K,m})_{ab}=e^{2\pi ij_a k_b/m}.$$
\end{defn}

The above understanding allows us to more casually write $H_{J,K,m}=(e^{2\pi i jk/m})_{j\in J,k\in K}$. When the context is clear, we will frequently write $H_{J,K}$ instead of $H_{J,K,m}$. 

Hadamard submatrices of the Fourier matrices are important to harmonic analysis researchers for multiple reasons. In \cite{JP98}, Jorgensen and Pedersen realized the quaternary Cantor set as the attractor of the affine iterated function system
$$\set{\tau_b(x):=\frac{x+b}{4}:b\in\set{0,2}},$$ which then gives rise to the quaternary Cantor measure $\mu_4$.
By using the fact that with $J=\set{0,2}$ and $K=\set{0,1}$, $H_{J,K}$ is a Hadamard submatrix of $\mathcal{F}_4$, they showed that $L^2(\mu_4)$ possesses an orthogonal basis of functions of the form $e^{2\pi i\lambda x}$. In fact, they showed that their argument can be conducted in general: If $m$ is a positive integer and $B\subseteq\set{0,1,\ldots{m-1}}$, then we may form the affine iterated function system $$\set{\tau_b(x):=\frac{x+b}{m}:b\in B},$$ which by Hutchinson's Theorem possesses a unique compact attractor set, and which gives rise to a unique invariant measure $\mu$ supported on the attractor. If a set $L\subseteq\set{0,1,\ldots,m-1}$ can be found, $\abs{L}=\abs{B}$, such that $H_{B,L}$ is a Hadamard submatrix of $\mathcal{F}_m$, then the set of complex exponential functions $\set{e^{2\pi i\lambda x}}_{\lambda\in \Lambda}$ is orthogonal in $L^2(\mu)$, where
$$\Lambda=\set{\sum_{k=0}^{K}\ell_km^k:K\in\mathbb{N}_0,\ell_k\in L}.$$

Such a compatible $m$, $B$, and $L$ as this, where $H_{B,L,m}$ is Hadamard, creates what is called a Hadamard triple: $(m,B,L).$ Determining when such an $L$ can be found, and when it cannot, is thus important to the search for orthogonal bases of complex exponential functions. For more results concerning the connection of Hadamard triples to the construction of spectral measures, we refer the reader to \cite{DHL19}.

Another reason Hadamard submatrices of the Fourier matrices are important is because of their connection to the Fuglede Conjecture. The Fuglede Conjecture posits that a subset $\Omega$ of $\mathbb{R}^n$ is spectral with respect to Lebesgue measure, meaning there exists an orthogonal basis of functions of the form $e^{2\pi i\langle\vec{x},\vec{\lambda}\rangle}$, if and only if $\Omega$ tiles $\mathbb{R}^n$.  While the conjecture holds for certain special cases, Tao proved in \cite{Tao04} that the general version is false in dimensions 5 and higher. Kolountzakis and Matolcsi later extended this result in \cite{KM06}, showing the Fuglede Conjecture is false in dimensions 3 and higher. However, the conjecture remains unresolved in both directions in dimensions 1 and 2. 

In \cite{DJ13}, Dutkay and Jorgensen showed that the forward direction of the Fuglede conjecture in dimension 1 is equivalent to a Universal Tiling Conjecture (UTC). The UTC conjectures that if equally-sized sets of integers are spectral with respect to the counting measure and share the same spectrum, then they will tile the integers using the same translations. For a finite set of integers $A$, a set $\Lambda\subseteq\mathbb{Q}$, $\abs{\Lambda}=\abs{A}$, will be a spectrum for $A$ if and only if $(m,m\Lambda,A)$ is a Hadamard triple, where $m$ is an integer such that $m\Lambda\subseteq\mathbb{Z}$. Thus, knowing what Hadamard triples exist is 
important towards proving or disproving the UTC. Determining whether and by what translations a set of integers will tile $\mathbb{Z}$ is another problem, for which we direct the reader to \cite{MR1670646} and related works.

It is not difficult to see that Fourier matrices must contain Hadamard submatrices of certain orders. In particular, if $n\mid m$, then letting $J=\set{0,1,2,\ldots,n-1}$ and $K=\frac{m}{n}J$, we see that $H_{J,K}=\mathcal{F}_n$ is a submatrix of $\mathcal{F}_m$, and so $\mathcal{F}_m$ contains a Hadamard submatrix of order $n$. Moreover, this easy construction gives us a Hadamard triple: $(m,J,K)$ (and of course $(m,K,J)$ is also a Hadamard triple).

It is also not difficult to see that Hadamard submatrices occur in much greater abundance than this. For example, in \cite{Tad06}, Tadej showed that Fourier matrices are equivalent to Kronecker products of Fourier matrices of lower order. If $m=a_1\cdot a_2\cdots a_k$, where the $a_i$'s are relatively prime, then there exist permutation matrices $P_r$ and $P_c$ such that
\begin{align}\label{TadejDecomp}\mathcal{F}_m=P_r(\mathcal{F}_{a_1}\otimes\mathcal{F}_{a_2}\otimes\cdots\otimes\mathcal{F}_{a_k})P_c.
\end{align}
The permutation matrices $P_r$ and $P_c$ themselves are constructed by successive computation of B\'{e}zout coefficients. From any such factorization, it can be seen that each of the Fourier matrices in the decomposition, as well as any Hadamard submatrices within them, gives rise to equivalent Hadamard submatrices in $\mathcal{F}_m$. 

However, these factorizations do not fully settle the question of the Hadamard submatrix structure of all Fourier matrices. While the Kronecker decomposition certainly demonstrates the existence of a number of Hadamard submatrices, it does not exclude the possibility of other Hadamard submatrices arising synergistically. In \cite{BGH19}, the question was asked for which $n$ the possibility of any $n\times n$ Hadamard submatrix existing could be excluded. While some results were obtained, it is still an open question, as far as we are aware, whether there exists a Fourier matrix $\mathcal{F}_m$ that contains an $n\times n$ Hadamard submatrix where $n$ does not divide $m$. It is also possible for there to exist Hadamard submatrices whose orders do divide $m$, but whose existence is not attributable to a Fourier matrix of lower order in a Kronecker decomposition, as exampled here:

\begin{prop}Not all Hadamard submatrices of Fourier matrices are equivalent to Fourier matrices of lower order. Moreover, neither the presence nor the non-presence of a lower-order Fourier matrix $\mathcal{F}_n$ in a Kronecker decomposition of $\mathcal{F}_m$ as in $(\ref{TadejDecomp})$ in general implies all $n\times n$ Hadamard submatrices are equivalent to $\mathcal{F}_n$.
\end{prop}

\begin{proof}We prove this by demonstrating some examples. As seen in \cite{TZ06}, which references \cite{Ha97}, it is known that every $4\times 4$ Hadamard matrix is equivalent to one matrix in the family $\mathcal{G}_4=\set{F_4^{(1)}(a):a\in[0,1/2)}$, where $$F_4^{(1)}(a)=\begin{bmatrix}1 & 1 & 1 & 1\\1 & i\cdot e^{2\pi i a} & -1 & -i\cdot e^{2\pi ia}\\ 1 & -1 & 1 & -1 \\1 & -i\cdot e^{2\pi ia} & -1 & i\cdot e^{2\pi ia} \end{bmatrix}.$$
Consider the Fourier matrix $\mathcal{F}_{64}$. Note that since $64=2^6$, $\mathcal{F}_{64}$ does not decompose as in (\ref{TadejDecomp}) into Fourier matrices of lower order. Let $J=\set{0,3,12,15}$ and $K=\set{0,8,32,40}$. Then
$$H_{J,K}=\begin{bmatrix}1 & 1 & 1 & 1\\1 & e^{2\pi i24/64} & e^{2\pi i96/64} & e^{2\pi i120/64}\\1 & e^{2\pi i96/64} & e^{2\pi i384/64} & e^{2\pi i 480/64}\\1 & e^{2\pi i120/64} & e^{2\pi i480/64} & e^{2\pi i 600/64}
\end{bmatrix}=\begin{bmatrix}1 & 1 & 1 & 1\\1 & e^{2\pi i3/8} & -1 & e^{2\pi i7/8}\\1 & -1 & 1 & -1\\1 & e^{2\pi i7/8} & -1 & e^{2\pi i 3/8}
\end{bmatrix}=F_4^{(1)}(1/8).$$
On the other hand, $\mathcal{F}_4=F_4^{(1)}(0)$. Thus, $H_{J,K}$ is a $4\times 4$ Hadamard submatrix of $\mathcal{F}_{64}$ that is not equivalent to $\mathcal{F}_4$.

In the above example, $\mathcal{F}_4$ was not a factor in a Kronecker decomposition of $\mathcal{F}_{64}$ as in (\ref{TadejDecomp}). We now give a second example in which it is. Consider $\mathcal{F}_{60}$. By \cite{Tad06}, there exist permutation matrices $P_r$ and $P_c$ such that $$\mathcal{F}_{60}=P_r(\mathcal{F}_3\otimes\mathcal{F}_4\otimes\mathcal{F}_5)P_c.$$ Let $J=\set{0,1,6,7}$ and $K=\set{0,25,30,55}$. Then
$$H_{J,K}=\begin{bmatrix}1 & 1 & 1 & 1\\1 & e^{2\pi i25/60} & e^{2\pi i30/60} & e^{2\pi i55/60}\\1 & e^{2\pi i150/60} & e^{2\pi i180/60} & e^{2\pi i 330/60}\\1 & e^{2\pi i175/60} & e^{2\pi i210/60} & e^{2\pi i 385/60}
\end{bmatrix}=\begin{bmatrix}1 & 1 & 1 & 1\\1 & e^{2\pi i5/12} & -1 & e^{2\pi i11/12}\\1 & -1 & 1 & -1\\1 & e^{2\pi i11/12} & -1 & e^{2\pi i 5/12}
\end{bmatrix}=F_4^{(1)}(1/6).$$
Thus, $H_{J,K}$ is a $4\times 4$ Hadamard submatrix of $\mathcal{F}_{60}$ that is not equivalent to $\mathcal{F}_4$, even though $\mathcal{F}_4$ appears in a Kronecker decomposition of $\mathcal{F}_{60}$.
\end{proof}

Since not all Hadamard submatrices of Fourier matrices are equivalent to Fourier matrices, they cannot be fully predicted by Kronecker decompositions. Moreover, even if all the Hadamard submatrices of a particular Fourier matrix are equivalent to Fourier matrices, it does not follow $\textit{a priori}$ that they all correspond to factors in a Kronecker decomposition. 

This paper does not fully determine the Hadamard submatrix structure or resolve all unresolved questions either, but it hopes to contribute to their eventual resolution by providing alternative methods for obtaining and representing information about the Hadamard submatrix structure. These methods are based on computing an attribute based on the spacing of a row or column selection called a primitive set. By checking one row and column combination of a Fourier matrix, we will immediately know whether all other same-sized row and column combinations with the same primitive sets are Hadamard. This works even if those other row and column selections have not been or cannot be traced backwards through a Kronecker decomposition to a lower-order Fourier matrix. Moreover, when compatible primitive sets are found for one Fourier matrix, those primitive sets must be non-compatible for submatrices of any other size, even those in a Fourier matrix of a different order. 

We require some additional notation and results before we begin:

\begin{defn}
A selection of rows $J\subseteq\set{0,1,\ldots,m-1}$ from $\mathcal{F}_m$ will be associated with the polynomial
$$J(z)=\sum_{j\in J}z^j,$$
and likewise a selection of columns $K$ will be associated with the polynomial
$$K(z)=\sum_{k\in K}z^k.$$
\end{defn}

We will refer to the $s$th cyclotomic polynomial by $\Phi_s(z)$. Recall that $\Phi_s(z)$ is the minimal polynomial of the primitive $s$th roots of unity.

The following special sets and numbers will aid us throughout the rest of the paper:

\begin{defn}Let $X$ be a nonempty, finite subset of $\mathbb{N}_0$, and let $m\in\mathbb{N}$. Define the difference set of $X$ by $$\mathcal{D}(X):=\set{x_1-x_2:x_1,x_2\in X}.$$ Define the $m$-th primitive set of $X$ by $$\mathcal{P}_m(X):=\set{\frac{m}{\text{gcd}(m,d)}:d\in\mathcal{D}(X)}.$$ Define $$C_m(X):=\prod_{s\in\mathcal{P}_m(X)\setminus\set{1}}\Phi_s(1).$$\end{defn}

Since $$\Phi_s(1)=\begin{cases}0 & \text{if }s=1     \\p & \text{if }s=p^k,\text{ where }p\text{ is a prime}, k\in\mathbb{N}\\
        1 & \text{ otherwise}\end{cases},$$ $C_m(J)$ is a product of the prime bases of each element of $\mathcal{P}_m(X)$ that is a power of a prime. 
				
Note also that $0\in\mathcal{D}(X)$ for any $X$, and therefore $1\in\mathcal{P}_m(X)$ for any $m$ and $X$.

We will frequently use the fact that rigidly shifting the selected rows $J$ or $K$ of $\mathcal{F}_m$ does not affect whether the submatrix $H_{J,K}$ is Hadamard, nor does it change the $m$th primitive set of $J$ or $K$. We state these facts more precisely:

\begin{prop}\sloppy Suppose $J,K\subseteq\set{0,1,\ldots,m-1}$. Let $a,b\in\mathbb{Z}$, and let ${J^\prime=\set{j+a\mod m:j\in J}}$ and ${K^\prime=\set{k+b\mod m:k\in K}}$. If $H_{J,K}$ is a Hadamard submatrix of $\mathcal{F}_m$, then so is $H_{J^\prime,K^\prime}$. \end{prop}

The above proposition was proved in \cite{BGH19}, but it is also a simple exercise for the reader.

\begin{prop}\sloppy Let $J\subseteq\set{0,1,\ldots,m-1}$. Let $v$ be an integer. Let ${J^\prime=\set{j+v\mod{m}:j\in J}}$. Then $\mathcal{P}_m(J^\prime)=\mathcal{P}_m(J)$.
\end{prop}

\begin{proof}Let $s^\prime\in\mathcal{P}_m(J^\prime)$. Then there exist $j_1^\prime,j_2^\prime\in J^\prime$ such that $s^\prime=\frac{m}{\gcd(m,j_2^\prime-j_1^\prime)}.$ There must exist $j_1,j_2\in J$ and integers $a$ and $b$ such that $j_1^\prime=j_1+v+am$ and $j_2^\prime=j_2+v+bm$. Then $$s^\prime=\frac{m}{\gcd(m,j_2+v+bm-j_1-v-am)}=\frac{m}{\gcd(m,j_2-j_1+(b-a)m)}.$$ It is a basic fact of number theory that for any integers $r$, $R$, and $c$, $\gcd(r,R)=\gcd(r,R+cr)$. Therefore,
$$s^\prime=\frac{m}{\gcd(m,j_2-j_1)}\in\mathcal{P}_m(J).$$ Hence, $\mathcal{P}_m(J^\prime)\subseteq\mathcal{P}_m(J)$. Since $J$ is obtained from $J^\prime$ by shifting every entry by $-v$, a symmetrical argument shows that $\mathcal{P}_m(J)\subseteq\mathcal{P}_m(J^\prime)$, and so $\mathcal{P}_m(J^\prime)=\mathcal{P}_m(J)$.
\end{proof}

\begin{defn}Let $p$ be a prime and $n$ be a nonzero integer. By $\nu_p(n)$ we denote the $p$-adic order of $n$. That is, $$\nu_p(n):=\max\set{v\in\mathbb{N}_0:p^v\text{ divides }n}.$$ If $X$ is a set of nonzero integers, then we define
$$\nu_{p}^\text{max}(X)=\max\set{\nu_p(n):n\in X}$$ and $$\nu_{p}^\text{min}(X)=\min\set{\nu_p(n):n\in X}.$$
\end{defn}

\begin{prop}\label{compprop}Let $p$ be a prime and let $X\subseteq\set{0,1,2,\ldots,m-1}$, with $\abs{X}\geq2$. Then,
\begin{align*}\nu_p^\text{max}(\mathcal{P}_m(X)\setminus\set{1})&\leq\max\set{0,\nu_p(m)-\nu_p^\text{min}(\mathcal{D}(X)\setminus\set{0})},\\
\nu_p^\text{min}(\mathcal{P}_m(X)\setminus\set{1})&\geq\nu_p(m)-\nu_p^\text{max}(\mathcal{D}(X)\setminus\set{0}),\text{ and}\\
\nu_p^\text{min}(\mathcal{D}(X)\setminus\set{0})&\geq\nu_p(m)-\nu_p^\text{max}(\mathcal{P}_m(X)\setminus\set{1}).\end{align*}
Furthermore, if $\nu_p^\text{min}(\mathcal{P}_m(X)\setminus\set{1})\geq1$, then
$$\nu_p^\text{max}(\mathcal{D}(X)\setminus\set{0})\leq\nu_p(m)-\nu_p^\text{min}(\mathcal{P}_m(X)\setminus\set{1}).$$
\end{prop}

\begin{proof}First, note that since $\abs{X}\geq2$, $\mathcal{D}(X)\setminus\set{0}$ is nonempty. If $d\in\mathcal{D}(X)\setminus\set{0}$, then $-m< d< m$ and $d\neq0$, and so $\frac{m}{\gcd(m,d)}\neq 1$. It follows that $\mathcal{P}_m(X)\setminus\set{1}$ is nonempty.

Let $s\in\mathcal{P}_m(X)\setminus\set{1}$. Let $d\in\mathcal{D}(X)\setminus\set{0}$ such that $s=\frac{m}{\gcd(m,d)}$. Then
\begin{align*}\nu_p(s)&=\nu_p(m)-\nu_p(\gcd(m,d))\\
&=\nu_p(m)-\min\set{\nu_p(m),\nu_p(d)}\\
&\leq\nu_p(m)-\min\set{\nu_p(m),\nu_p^\text{min}(\mathcal{D}(X)\setminus\set{0})}\\
&=\max\set{0,\nu_p(m)-\nu_p^\text{min}(\mathcal{D}(X)\setminus\set{0})},\end{align*}
and
\begin{align*}\nu_p(s)&=\nu_p(m)-\nu_p(\gcd(m,d))\\
&=\nu_p(m)-\min\set{\nu_p(m),\nu_p(d)}\\
&\geq\nu_p(m)-\nu_p(d)\\
&\geq\nu_p(m)-\nu_p^\text{max}(\mathcal{D}(X)\setminus\set{0}).\end{align*}
Since $s$ was arbitrary, 
$$\nu_p^\text{max}(\mathcal{P}_m(X)\setminus\set{1})\leq\text{max}\set{0,\nu_p(m)-\nu_p^\text{min}(\mathcal{D}(X)\setminus\set{0})},$$ and
$$\nu_p^\text{min}(\mathcal{P}_m(X)\setminus\set{1})\geq\nu_p(m)-\nu_p^\text{max}(\mathcal{D}(X)\setminus\set{0}).$$

Now let $d\in\mathcal{D}(X)\setminus\set{0}$. Let $s=\frac{m}{\gcd(m,d)}$. Then $s\in\mathcal{P}_m(X)\setminus\set{1}$. Observe that
\begin{align*}\nu_p^\text{max}(\mathcal{P}_m(X)\setminus\set{1})&\geq\nu_p(s)\\
&=\nu_p(m)-\nu_p(\gcd(m,d))\\
&\geq\nu_p(m)-\nu_p(d).
\end{align*}
Therefore, since $d$ was arbitrary, $$\nu_p^\text{min}(\mathcal{D}(X)\setminus\set{0})\geq\nu_p(m)-\nu_p^\text{max}(\mathcal{P}_m(X)\setminus\set{1}).$$

Now, suppose $\nu_p^\text{min}(\mathcal{P}_m(X)\setminus\set{1})\geq1$. Let $d\in\mathcal{D}(X)\setminus\set{0}$. Let $s=\frac{m}{\gcd(m,d)}$. Then $s\in\mathcal{P}_m(X)\setminus\set{1}$, and by assumption, $1\leq\nu_p(s)=\nu_p(m)-\nu_p(\gcd(m,d))$. It follows that $\nu_p(\gcd(m,d))<\nu_p(m)$. Since $\nu_p(\gcd(m,d))=\min\set{\nu_p(m),\nu_p(d)}$, this implies $\nu_p(\gcd(m,d))=\nu_p(d)$. Hence,
\begin{align*}\nu_p(d)&=\nu_p(\gcd(m,d))\\
&=\nu_p(m)-\nu_p(s)\\
&\leq\nu_p(m)-\nu_p^\text{min}(\mathcal{P}_m(X)\setminus\set{1}).
\end{align*}
Therefore, since $d$ was arbitrary,
$$\nu_p^\text{max}(\mathcal{D}(X)\setminus\set{0})\leq\nu_p(m)-\nu_p^\text{min}(\mathcal{P}_m(X)\setminus\set{1}).$$
\end{proof}

\section{Main Results}

\begin{thm}\label{Ctheorem}If $C_m(J)$ does not divide $\abs{J}$, then $H_{J,K}$ cannot be a Hadamard submatrix of $\mathcal{F}_m$ for any $K$.
\end{thm}

\begin{proof}Suppose, for the sake of contradiction, that $H_{J,K}$ were Hadamard. Let $s\in\mathcal{P}_m(J)\setminus\set{1}$. Then there exists a $d\in\mathcal{D}(J)$ such that $s=\frac{m}{\gcd(m,d)}$, and since $s\neq1$, $d\neq0$. It follows that there exist distinct rows $j_1,j_2\in J$ such that $d=j_1-j_2$. Since distinct rows of a Hadamard matrix are orthogonal, we have
    $$\sum_{k\in K}e^{2\pi idk/m}=0.$$
    This implies that $e^{2\pi id/m}$ is a root of $K(z)$. Since $e^{2\pi id/m}$ is a primitive $s$-th root of unity, it follows that the cyclotomic polynomial $\Phi_s(z)$ divides $K(z)$.
    
    Therefore there exists some polynomial $p(z)$ with integer coefficients such that
    $$K(z)=p(z)\prod_{s\in\mathcal{P}_m(J)\setminus\set{1}}\Phi_s(z).$$
    However, this means that 
    \begin{align*}\abs{J} & =\abs{K}                                                   \\
                & =K(1)                                                      \\
                & =p(1)\prod_{s\in\mathcal{P}_m(J)\setminus\set{1}}\Phi_s(1) \\
                & =p(1)C_m(J).                                                \\
    \end{align*}
    This contradicts the fact that $C_m(J)$ does not divide $\abs{J}$.
\end{proof}

\textbf{Example:} Consider $\mathcal{F}_{6000}$, and let $J=\set{0,5,375}$. We have $\mathcal{D}(J)=\set{0,\pm5,\pm370,\pm375}$, and $\mathcal{P}_{6000}(J)=\set{1,16,600,1200}$. Then $C_{6000}(J)=2\cdot1\cdot1=2$. Since $2$ does not divide $3$, by Theorem \ref{Ctheorem}, $H_{J,K}$ (and by symmetry, $H_{K,J}$) is not a Hadamard submatrix of $\mathcal{F}_{6000}$ for any $K$.

\begin{cor}\label{Ctheoremcor1}If $C_m(J)>\abs{J}$, then $H_{J,K}$ cannot be a Hadamard submatrix of $\mathcal{F}_m$ for any $K$.
\end{cor}

\begin{proof}If $C_m(J)>\abs{J}$, then $C_m(J)$ does not divide $\abs{J}$. So by Theorem \ref{Ctheorem}, $H_{J,K}$ is not Hadamard.\end{proof}

\textbf{Example:} Consider $\mathcal{F}_6$, and let $J=\set{0,4}$. We have $\mathcal{D}(J)=\set{0,\pm4}$, and $\mathcal{P}_6(J)=\set{1,3}$. Therefore, $C_6(J)=3>\abs{J}=2$. It follows that $H_{J,K}$ (and by symmetry, $H_{K,J}$) is not a Hadamard submatrix of $\mathcal{F}_6$ for any $K$.

\begin{cor}\label{Ctheoremcor2}If $\mathcal{P}_m(J)$ contains all the prime power factors of $m$ but lacks at least one positive factor of $m$, then $H_{J,K}$ (and by symmetry, $H_{K,J}$) is not a Hadamard submatrix of $\mathcal{F}_m$ for any $K$.
\end{cor}
\begin{proof}If $\mathcal{P}_m(J)$ contains all the prime power factors of $m$, then $C_m(J)=m$. Therefore, if $H_{J,K}$ were Hadamard, then $\abs{J}\geq C_m(J)=m$ by Corollary \ref{Ctheoremcor1}, implying that in fact $J=\set{0,1,\ldots,m-1}$. This means $\mathcal{P}_m(J)$ contains all the positive factors of $m$, which is a contradiction.
\end{proof}

\textbf{Example:} Consider $\mathcal{F}_{12}$, and let $J=\set{0,1,6,9}$. Then $\mathcal{D}(J)=\set{0,\pm1,\pm3,\pm5\pm6,\pm8,\pm9}$, and $\mathcal{P}_{12}(J)=\set{1,2,3,4,12}$. So $\mathcal{P}_{12}(J)$ contains all the prime power factors of $m=12$, namely $2$, $3$, and $4$, but it lacks the factor $6$. Hence, by Corollary \ref{Ctheoremcor2}, $H_{J,K}$ (and by symmetry, $H_{K,J}$) is not a Hadamard submatrix of $\mathcal{F}_{12}$ for any $K$.

\begin{thm}Let $J,K\subseteq\set{0,1,\ldots,m-1}$, and let $A$ be a set of nonnegative integers such that $K\oplus A$ contains exactly one representative from each congruence class modulo $m$. If for all $s\in\mathcal{P}_m(J)\setminus\set{1}$, $\Phi_s(z)$ fails to divide $A(z)$, then $H_{J,K}$ is a Hadamard submatrix of $\mathcal{F}_m$.
\end{thm}

\begin{proof}
    Let $j_1,j_2$ be distinct elements of $J$. Since the Fourier matrix $\mathcal{F}_m$ is Hadamard, we have
    \begin{align*}0 & =\sum_{k=0}^{m-1}e^{2\pi i(j_1-j_2)k/m}                                         \\
          & =\sum_{a\in A}\sum_{k\in K}e^{2\pi i(j_1-j_2)(k+a)/m}                           \\
          & =\sum_{a\in A}\p{e^{2\pi i a(j_1-j_2)/m}\sum_{k\in K}e^{2\pi i(j_1-j_2)k/m}}    \\
          & =\p{\sum_{a\in A}e^{2\pi ia(j_1-j_2)/m}}\p{\sum_{k\in K}e^{2\pi i(j_1-j_2)k/m}}.
    \end{align*}
    Now, $e^{2\pi i(j_1-j_2)/m}$ is a primitive $s$-th root of unity for some $s\in\mathcal{P}_m(J)\setminus\set{1}$. Since $\Phi_s(z)$ does not divide $A(z)$, it follows that $e^{2\pi i(j_1-j_2)/m}$ is not a root of $A(z)$, and hence $$\sum_{a\in A}e^{2\pi ia(j_1-j_2)/m}\neq0.$$ It follows that
    $$\sum_{k\in K}e^{2\pi i(j_1-j_2)k/m}=0,$$ which completes the proof.
\end{proof}

\textbf{Example:} Let $m=10$, $J=\set{0,1,7,8,9}$, and $K=\set{0,2,4,6,8}$. If we let $A=\set{0,1}$, then $K\oplus A=\set{0,1,2,3,4,5,6,7,8,9}$, which contains exactly one representative of each congruence class modulo $10$. Note that $\mathcal{D}(J)=\set{0,\pm1,\pm2,\pm6,\pm7,\pm8,\pm9}$, and so $\mathcal{P}_{10}(J)=\set{1,5,10}$. We have that $\Phi_5(z)=z^4+z^2+z+1$ and $\Phi_{10}(z)=z^4-z^3+z^2-z+1$, neither of which divide $A(z)=1+z$. Hence, $H_{J,K}$ is a Hadamard submatrix of $\mathcal{F}_{10}$. In fact, the only cyclotomic polynomial that divides $A(z)$ is $\Phi_2(z)=A(z)$. Therefore, $H_{J,K}$ would be Hadamard for any other $J$ so long as $2\not\in\mathcal{P}_{10}(J)$.

\begin{lem}\label{tielemma}Let $J,K\subseteq\set{0,1,\ldots,m-1}$ with $\abs{J}=\abs{K}$ such that $H_{J,K}$ is a Hadamard submatrix of $\mathcal{F}_m$. Suppose that $J^\prime\subseteq\set{0,1,\ldots,m-1}$ with $\abs{J^\prime}=\abs{K}$. If $\mathcal{P}_m(J^\prime)=\mathcal{P}_m(J)$, then $H_{J^\prime,K}$ is also Hadamard submatrix of $\mathcal{F}_m$. Alternatively, if $K^\prime\subseteq\set{0,1,\ldots,m-1}$ with $\abs{K^\prime}=\abs{J}$, then $\mathcal{P}_m(K^\prime)=\mathcal{P}_m(K)$ implies $H_{J,K^\prime}$ is also a Hadamard submatrix of $\mathcal{F}_m$.
\end{lem}

\begin{proof}Let $j_1^\prime,j_2^\prime\in J^\prime$ be distinct, and let $d^\prime=j_1^\prime-j_2^\prime$. Then $d^\prime\in\mathcal{D}(J^\prime)$. Let $s=\frac{m}{\gcd(m,d^\prime)}$. Then $s\in\mathcal{P}_m(J^\prime)$, and so $s\in\mathcal{P}_m(J)$. It follows that there exist $j_1,j_2\in J$ such that $s=\frac{m}{\gcd(m,j_1-j_2)}$. Since $H_{J,K}$ is Hadamard, we have
    $$\sum_{k\in K}e^{2\pi i(j_1-j_2)k/m}=0.$$
    Thus, $e^{2\pi i(j_1-j_2)/m}$ is a root of $K(z)$. By construction, it is a primitive $s$-th root, and so the cyclotomic polynomial $\Phi_s(z)$ divides $K(z)$. Then because $e^{2\pi i(j_1^\prime-j_2^\prime)/m}$ is a primitive $s$-th root of unity, it is also a root of $\Phi_s(z)$ and hence a root of $K(z)$. Therefore,
    $$\sum_{k\in K}e^{2\pi i(j_1^\prime-j_2^\prime)k/m}=0.$$
    This shows that $H_{J^\prime,K}$ is Hadamard.
    By a symmetric argument, if $K^\prime\subseteq\set{0,1,\ldots,m-1}$ with $\abs{K^\prime}=\abs{J}$ and $\mathcal{P}_m(K^\prime)=\mathcal{P}_m(K)$, then $H_{J,K^\prime}$ is a Hadamard submatrix of $\mathcal{F}_m$.
\end{proof}

\begin{thm}\label{equivtheorem}Let $J,J^\prime,K,K^\prime\subseteq\set{0,1,\ldots,m-1}$ with $\abs{J}=\abs{J^\prime}=\abs{K}=\abs{K^\prime}$, and let $H_{J,K}$ be a Hadamard submatrix of $\mathcal{F}_m$. If $\mathcal{P}_m(J^\prime)=\mathcal{P}_m(J)$ and $\mathcal{P}_m(K^\prime)=\mathcal{P}_m(K)$, then $H_{J^\prime,K^\prime}$ is also a Hadamard submatrix of $\mathcal{F}_m$. Alternatively, if $\mathcal{P}_m(J^\prime)=\mathcal{P}_m(K)$ and $\mathcal{P}_m(K^\prime)=\mathcal{P}_m(J)$, then $H_{J^\prime,K^\prime}$ is also a Hadamard submatrix of $\mathcal{F}_m$.
\end{thm}

\begin{proof}Suppose $\mathcal{P}_m(J^\prime)=\mathcal{P}_m(J)$ and $\mathcal{P}_m(K^\prime)=\mathcal{P}_m(K)$. By Lemma \ref{tielemma}, because $H_{J,K}$ is Hadamard, $H_{J^\prime,K}$ is Hadamard. Then because $H_{J^\prime,K}$ is Hadamard, $H_{J^\prime,K^\prime}$ is Hadamard. Alternatively, suppose $\mathcal{P}_m(J^\prime)=\mathcal{P}_m(K)$ and $\mathcal{P}_m(K^\prime)=\mathcal{P}_m(J)$. By Lemma \ref{tielemma}, since $H_{J,K}$ is Hadamard, $H_{J,J^\prime}$ is Hadamard, and because $H_{J,J^\prime}$ is Hadamard, $H_{K^\prime,J^\prime}$ is a Hadamard submatrix of $\mathcal{F}_m$. Since the transpose of a Hadamard matrix is Hadamard, it follows that by selecting rows $J^\prime$ and columns $K^\prime$ from $\mathcal{F}_m^T$, the resulting submatrix is Hadamard, and since $\mathcal{F}_m^T=\mathcal{F}_m$, the proof is complete.
\end{proof}

Theorem \ref{equivtheorem} shows that for a given dimension $m$ and subdimension $\abs{J}=\abs{K}=n$, whether the submatrix $H_{J,K}$ of $\mathcal{F}_m$ is Hadamard is determined exactly by whether the primitive sets $\mathcal{P}_m(J)$ and $\mathcal{P}_m(K)$ are compatible. We can represent this compatibility structure by a graph:

\begin{defn}Define the compatibility graph $G(m,n)$ as follows: The vertices of $G(m,n)$ are the primitive sets of the row selections $J$ for which there is a Hadamard submatrix $H_{J,K}$ of $\mathcal{F}_m$:
    $$V(G(m,n))=\set{\mathcal{P}_m(J):H_{J,K}\text{ is an $n\times n$ Hadamard submatrix of }\mathcal{F}_m}.$$
    (By symmetry, this also includes the primitive sets of the column selections.)
    The edge set of $G(m,n)$ consists of edges connecting those vertices for which the primitive sets are compatible. That is,
    $$E(G(m,n))=\set{\set{\mathcal{P}_m(J),\mathcal{P}_m(K)}:H_{J,K}\text{ is an $n\times n$ Hadamard submatrix of $\mathcal{F}_m$}}.$$
    The graph is undirected but may contain loops. The graph will be empty if there are no Hadamard submatrices for a given $m$ and $n$.
\end{defn}

Note that since two primitive sets are either compatible or not, all $n\times n$ square submatrices of a Fourier matrix $\mathcal{F}_m$ may be partitioned into equivalence classes, with each equivalence class consisting of submatrices coming from row and column sets having the same two primitive sets. Each such equivalence class consists entirely of submatrices that are Hadamard or entirely of submatrices that are not Hadamard. The equivalence classes that do consist of Hadamard submatrices are then represented by the edges in the compatibility graph $G(m,n)$.

\begin{thm}\label{mutexcthm}Let $J,K\subseteq\set{0,1,\ldots,m-1}$ with $\abs{J}=\abs{K}$ be such that $H_{J,K}$ is a Hadamard submatrix of $\mathcal{F}_m$. Let $L\subseteq\set{0,1,\ldots,m-1}$ with $\abs{L}>\abs{J}$. Then $\mathcal{P}_m(L)\neq\mathcal{P}_m(J)$.
\end{thm}
\begin{proof}Assume, for the sake of contradiction, that $\mathcal{P}_m(L)=\mathcal{P}_m(J)$. Consider the (non-square) submatrix $H_{L,K}$. By reasoning akin to the proof of Lemma \ref{tielemma}, the rows of $H_{L,K}$ are mutually orthogonal. This is a contradiction, because there cannot be $\abs{L}$ orthogonal vectors of length $\abs{K}$ when $\abs{K}<\abs{L}$.
\end{proof}

\begin{cor}\label{DisjVertForN}If $n\neq n^\prime$, then $V(G(m,n))\cap V(G(m,n^\prime))=\varnothing$.
\end{cor}
\begin{proof}Without loss of generality, assume $n<n^\prime$. Let $P\in V(G(m,n))$. Then $P=\mathcal{P}_m(J)$ for some $J,K\subseteq\set{0,1,\ldots,m-1}$ with $\abs{J}=\abs{K}=n$ and $H_{J,K}$ a Hadamard submatrix of $\mathcal{F}_m$. By Theorem \ref{mutexcthm}, $P\neq\mathcal{P}_m(L)$ for any $L\subseteq\set{0,1\ldots,m-1}$ of cardinality larger than $n$, and so certainly $P\not\in V(G(m,n^\prime))$.\end{proof}

In other words, for a given $m$, an $m$-th primitive set can correspond to row or column sets of Hadamard submatrices of $\mathcal{F}_m$ of at most one particular size. Moreover, once a set $J$ is the row (or column) set of a Hadamard submatrix of $\mathcal{F}_m$, no subset of $\set{0,1,\ldots,m-1}$ of greater cardinality can be found that has the same $m$-th primitive set as $J$, even if it is not the row set of a Hadamard submatrix.

We now show that primitive sets cannot be shared among Hadamard submatrices of different sizes even when using different $m$.

\begin{lem}\label{sameprimitive}Let $J\subseteq\set{0,1,\ldots,m-1}$ and let $v\in\mathbb{N}$. Then $\mathcal{P}_{vm}(vJ)=\mathcal{P}_m(J)$.
\end{lem}
\begin{proof}Let $s\in\mathcal{P}_{vm}(vJ)$. Then for some $j_1,j_2\in J$, $s=\frac{vm}{\gcd(vm,vj_1-vj_2)}=\frac{vm}{v\gcd(m,j_1-j_2)}=\frac{m}{\gcd(m,j_1-j_2)}\in\mathcal{P}_m(J)$. Let $s\in\mathcal{P}_{m}(J)$. Then the same chain of equalities in reverse shows that $s\in\mathcal{P}_{vm}(vJ)$.
\end{proof}

\begin{lem}\label{DivVertCont}Let $v\in\mathbb{N}$. Then $V(G(m,n))\subseteq V(G(vm,n))$.
\end{lem}

\begin{proof}Let $P\in V(G(m,n))$. Then there exist $J,K\subseteq\set{0,1,\ldots, m-1}$ with $\abs{J}=\abs{K}=n$ such that $H_{J,K,m}$ is a Hadamard submatrix of $\mathcal{F}_m$ and $P=\mathcal{P}_m(J)$. Observe that $$H_{vJ,K,vm}=\p{e^{2\pi ivjk/(vm)}}_{j\in J,k\in K}=\p{e^{2\pi ijk/m}}_{j\in J,k\in K}=H_{J,K,m}.$$ Thus, $H_{vJ,K,vm}$ is an $n\times n$ Hadamard submatrix of $\mathcal{F}_{vm}$. Hence, $\mathcal{P}_{vm}(vJ)\in V(G(vm,n))$. By Lemma \ref{sameprimitive}, $\mathcal{P}_{vm}(vJ)=\mathcal{P}_m(J)$. Thus, $P\in V(G(vm,n))$.
\end{proof}

\begin{lem}\label{PrimDifferent}\sloppy Suppose $J,K\subseteq\set{0,1,\ldots,m-1}$ with $\abs{J}=\abs{K}=n$, and $J^\prime,K^\prime\subseteq\set{0,1,\ldots,m^\prime-1}$ with $\abs{J^\prime}=\abs{K^\prime}=n^\prime$. Suppose $H_{J,K,m}$ and $H_{J^\prime,K^\prime,m^\prime}$ are Hadamard submatrices of $\mathcal{F}_m$ and $\mathcal{F}_{m^\prime}$, respectively. If $n\neq n^\prime$, then $\mathcal{P}_m(J)\neq\mathcal{P}_{m^\prime}(J^\prime)$.
\end{lem}
\begin{proof}\sloppy
    Suppose $n\neq n^\prime$. By Lemma \ref{DivVertCont}, $\mathcal{P}_{m}(J)\in V(G(mm^\prime,n))$ and $\mathcal{P}_{m^\prime}(J^\prime)\in V(G(mm^\prime,n^\prime))$. By Corollary \ref{DisjVertForN}, $V(G(mm^\prime,n))\cap V(G(mm^\prime,n^\prime))=\varnothing$. Hence, $P_{m}(J)\neq P_{m^\prime}(J^\prime)$. 
\end{proof}

\begin{thm}\label{PrimThm}
    If $n\neq n^\prime$, then $V(G(m,n))\cap V(G(m^\prime,n^\prime))=\varnothing$ for any $m,m^\prime\in\mathbb{N}$.
\end{thm}
\begin{proof}This is a direct consequence of Lemma \ref{PrimDifferent}.
\end{proof}

\textbf{Example: }Consider $m=21$, and take $J=\set{0,2,16}$ and $K=\set{0,7,14}$. The reader may confirm, or may check using Theorem \ref{3by3thm}, that $H_{J,K}$ is a Hadamard submatrix of $\mathcal{F}_{21}$. We have $\mathcal{D}(J)=\set{0,\pm2,\pm14,\pm16}$, and so $\mathcal{P}_{21}(J)=\set{1,3,21}$. We have $\mathcal{D}(K)=\set{0,\pm7,\pm14}$, and so $\mathcal{P}_{21}(K)=\set{1,3}$. From this, we know the following:
\begin{itemize}\item By Theorem \ref{mutexcthm}, there does not exist a subset $X$ of $\set{0,1,\ldots,20}$ of cardinality greater than 3 such that $\mathcal{P}_{21}(X)=\set{1,3,21}$ or $\mathcal{P}_{21}(X)=\set{1,3}$. 
\item However, Theorem \ref{mutexcthm} does not preclude there being a subset of lower cardinality having one of these as its primitive set. For example, $X=\set{0,7}$ has $\mathcal{D}(X)=\set{0,\pm7}$ and $\mathcal{P}_{21}(X)=\set{1,3}$.
\item By Theorem \ref{PrimThm}, any Hadamard submatrix of any Fourier matrix that has $\set{1,3}$ or $\set{1,3,21}$ as the primitive set of its row or column set must be of size $3\times 3$. For example, $\mathcal{P}_{12}(\set{0,4})=\set{1,3}$. Since $\set{0,4}$ is of size 2, we conclude it is not the row or column set of any Hadamard submatrix of $\mathcal{F}_{12}$.
\item However, Theorem \ref{PrimThm} does not preclude there being a Fourier matrix other than $\mathcal{F}_{21}$ with $\set{1,3}$ or $\set{1,3,21}$ as the primitive set of the row or column set of a Hadamard submatrix, so long as that Hadamard submatrix is of size $3\times 3$. For example, if we take $J=\set{0,4,8}$ and $K=\set{0,1,2}$, then $H_{J,K}$ is a $3\times3$ Hadamard submatrix of $\mathcal{F}_{12}$ with $\mathcal{P}_{12}(J)=\set{1,3}$. 
\end{itemize}

We note that Theorem $\ref{PrimThm}$ allows us to partition all finite subsets of $\mathbb{N}$ into equivalence classes, with two sets being equivalent if they are the primitive sets of the row or column sets of same-sized Hadamard submatrices of Fourier matrices (or are not the primitive sets of Hadamard submatrices of any size). Moreover, we may define a function $${\phi:\set{A\subset\mathbb{N}:\abs{A}<\infty}\to\mathbb{N}_0}$$ by $\phi(X)=n$ if $X$ is the primitive set of the row or column set of an $n\times n$ Hadamard submatrix of a Fourier matrix, and $\phi(X)=0$ otherwise. In the previous example, $\phi(\set{1,3})=\phi(\set{1,3,21})=3$.

\section{Characterizations of $G(m,2)$ and $G(m,3)$}

In the previous section, we showed that compatibility between the primitive sets of the selected rows and columns is what determines whether a submatrix of $\mathcal{F}_m$ is Hadamard, and we represented this structure by a graph $G(m,n)$. An $n\times n$ submatrix $H_{J,K}$ of $\mathcal{F}_m$ is Hadamard if and only if $\mathcal{P}_m(J)$ and $\mathcal{P}_m(K)$ are vertices in $G(m,n)$ with an edge between them. In this section, we give some results that facilitate construction of these graphs when $n=2$ or $n=3$.

Before giving results that work for $G(m,2)$ and $G(m,3)$ in general, we give a couple results for special cases of $n=2$ where the graph is very simply described. The following result gives a complete characterization of $G(2^q,2)$, where $q\in\mathbb{N}$:

\begin{thm}\label{2by2thm}
    Let $J,K\subseteq\{0,1,\ldots,2^q-1\}$, $\abs{J}=\abs{K}=2$. Then $H_{J,K}$ is Hadamard if and only if
     there exist $\alpha,\beta\in\mathbb{N}_0$ with $\alpha+\beta=q-1$ such that $\mathcal P_{2^q}(J) = \{ 1 , 2 ^{q-\alpha} \} $ and $\mathcal P_{2^q}(K) = \{ 1 , 2 ^{q-\beta} \} $.
\end{thm}

\begin{proof}
    Suppose $H_{J,K}$ is Hadamard. Because $H_{J,K}$ has mutually orthogonal rows, for any $d\in\mathcal{D}(J)\setminus\set{0}$, $$ \sum_{k \in K} e^{\frac{2\pi i d k}{2^q}} = 0.$$ 
		Without loss of generality, we may assume that $j_1 = k_1 = 0$, so that $\mathcal{D}(J)=\{0,j_2,-j_2\}$ and $\mathcal{D}(K)=\{0,k_2,-k_2\}$. Therefore, taking $d=j_2$, $$0 = 1 +  e^{\frac{2\pi i j_2 k_2}{2^q}}.$$ 
    Therefore, $\ds e^{\frac{2\pi i j_2 k_2}{2^q}} = -1$. 
    This occurs when $\frac{2 j_2 k_2}{2^q} \in \mathbb Z_{\text{odd}}$.
    Thus, $j_2 k_2 = y2^{q-1} , y \in \mathbb Z_{\text{odd}}$.
    Therefore, there exist $\alpha,\beta\in\mathbb{N}_0$, $\alpha+\beta=q-1$, and odd integers $y_j$ and $y_k$, such that $j_2=2^\alpha y_j$ and $k_2=2^{\beta}y_k$.
		Hence, \begin{align*}\frac{2^q}{\gcd(2^q,d)}&=\frac{2^q}{\gcd(2^q,j_2)}\\
		&=\frac{2^q}{\gcd(2^q,2^\alpha y_j)}\\
		&=\frac{2^q}{2^\alpha}\\
		&=2^{q-\alpha}\in\mathcal{P}_{2^q}(J).\end{align*}
    Likewise, $2^{q-\beta}\in\mathcal{P}_{2^q}(K)$. Taking $d=-j_2$ and $d=-k_2$ also results in $2^{q-\alpha}$ and $2^{q-\beta}$, respectively. Finally, taking $d=0$ shows that $1\in\mathcal{P}(J)$ and $1\in\mathcal{P}(K)$ as usual.
    Therefore, $\mathcal P_{2^q}(J) = \{ 1 , 2^{q-\alpha} \}$, and $\mathcal P_{2^q}(K) = \{ 1 , 2^{q-\beta} \}$.
		
		Conversely, suppose there exist $\alpha,\beta\in\mathbb{N}_0$, $\alpha+\beta=q-1$, such that $\mathcal{P}_{2^q}(J)=\{1,2^{q-\alpha}\}$ and $\mathcal{P}_{2^q}(K)=\{1,2^{q-\beta}\}$. Without loss of generality, we may assume $J=\{0,j_2\}$ and $K=\{0,k_2\}$. It follows that
		$2^{q-\alpha}=\frac{2^q}{\gcd(2^q,j_2)}$, implying that $\gcd(2^q,j_2)=2^{\alpha}$. Thus, $j_2=2^\alpha y_j$, where $y_j$ is odd. Similarly, $k_2=2^{\beta} y_k$, where $y_k$ is odd. Note that
		\begin{align*}e^{\frac{2\pi i(j_2-j_1)k_1}{2^q}}+e^{\frac{2\pi i(j_2-j_1)k_2}{2^q}}&=1+e^{\frac{2\pi i j_2k_2}{2^q}}\\
		&=1+e^{\frac{2\pi i 2^{\alpha+\beta}y_jy_k}{2^q}}\\
		&=1+e^{\pi i y_jy_k}\\
		&=1-1=0.\end{align*}
		Hence, the two rows of $H_{J,K}$ are orthogonal, implying that $H_{J,K}$ is Hadamard.
		
\end{proof}

\begin{cor}
    Let $q \in \mathbb{N}$. Then $|V(G(2^q,2))| = q$ and $|E(G(2^q,2))| = \lceil{\frac{q}{2}}\rceil.$
\end{cor}

\begin{proof}
		If $v\in V(G(2^q,2))$, then by Theorem $\ref{2by2thm}$, $v=\{1,2^{q-\alpha}\}$ for some $\alpha\in\mathbb{N}_0$, $0\leq\alpha\leq q-1$. Conversely, let $v=\{1,2^{q-\alpha}\}$ for some $\alpha\in\mathbb{N}_0$, $0\leq \alpha\leq q-1$. Let $\beta=q-\alpha-1$. Then $\alpha+\beta=q-1$. Let $J=\{0,2^\alpha\}$ and $K=\{0,2^{\beta}\}$. Then $\mathcal{P}_{2^q}(J)=\{1,2^{q-\alpha}\}$ and $\mathcal{P}_{2^q}(K)=\{1,2^{q-\beta}\}$. By Theorem \ref{2by2thm}, $H_{J,K}$ is a Hadamard submatrix of $\mathcal{F}_{2^q}$. Therefore, $v=\mathcal{P}_{2^q}(J)\in V(G(2^q,2))$. Thus, $V(G(2^q,2))=\{\{1,2^q-\alpha\}:\alpha\in\mathbb{N}_0,0\leq \alpha\leq q-1\}$, and so $\abs{V(G(2^q,2))}=q$.
		
		By Theorem \ref{2by2thm}, for any $0\leq\alpha\leq q-1$, $\{1,2^{q-\alpha}\}$ is connected only to $\{1,2^{\alpha+1}\}$. If $q$ is even, then $q-\alpha$ and $\alpha+1$ have opposite parity, so that $q-\alpha\neq\alpha+1$, and hence $\{1,2^{q-\alpha}\}$ and $\{1,2^{\alpha+1}\}$ are distinct vertices. Thus, $\abs{E(G(2^q,2))}=\frac{q}{2}$. If $q$ is odd, then $q-\alpha=\alpha+1$ only when $\alpha=\frac{q-1}{2}$. Hence, $\abs{E(G(2^q,2))}=\frac{q-1}{2}+1=\frac{q+1}{2}$. Therefore, $\abs{E(G(2^q,2))}=\lceil\frac{q}{2}\rceil$.
\end{proof}

Since any set of the form $\set{1,2^a}$, where $1\leq a\leq q$, can be realized as the primitive set of a 2-element set (namely, $\set{1,2^a}=\mathcal{P}_{2^q}(\set{0,2^{q-a}})$), Theorem \ref{2by2thm} implies that $G(2^q,2)$ is formed simply by connecting vertices of the form $\set{1,2^a}$, $1\leq a\leq q$, where the powers on the 2 add up to $q+1$. For example, here are the graphs of $G(16,2)$ and $G(32,2)$:
\begin{figure}[H]
\includegraphics[scale=.25]{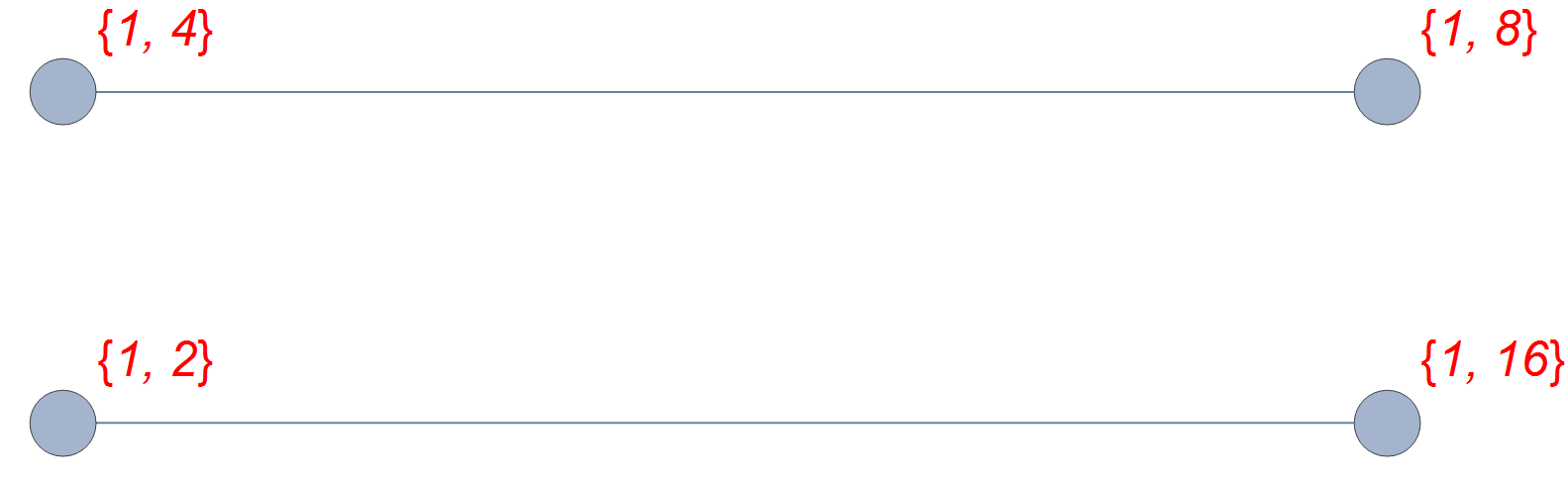}
\caption{G(16,2)}
\end{figure}
\begin{figure}[H]\includegraphics[scale=.20]{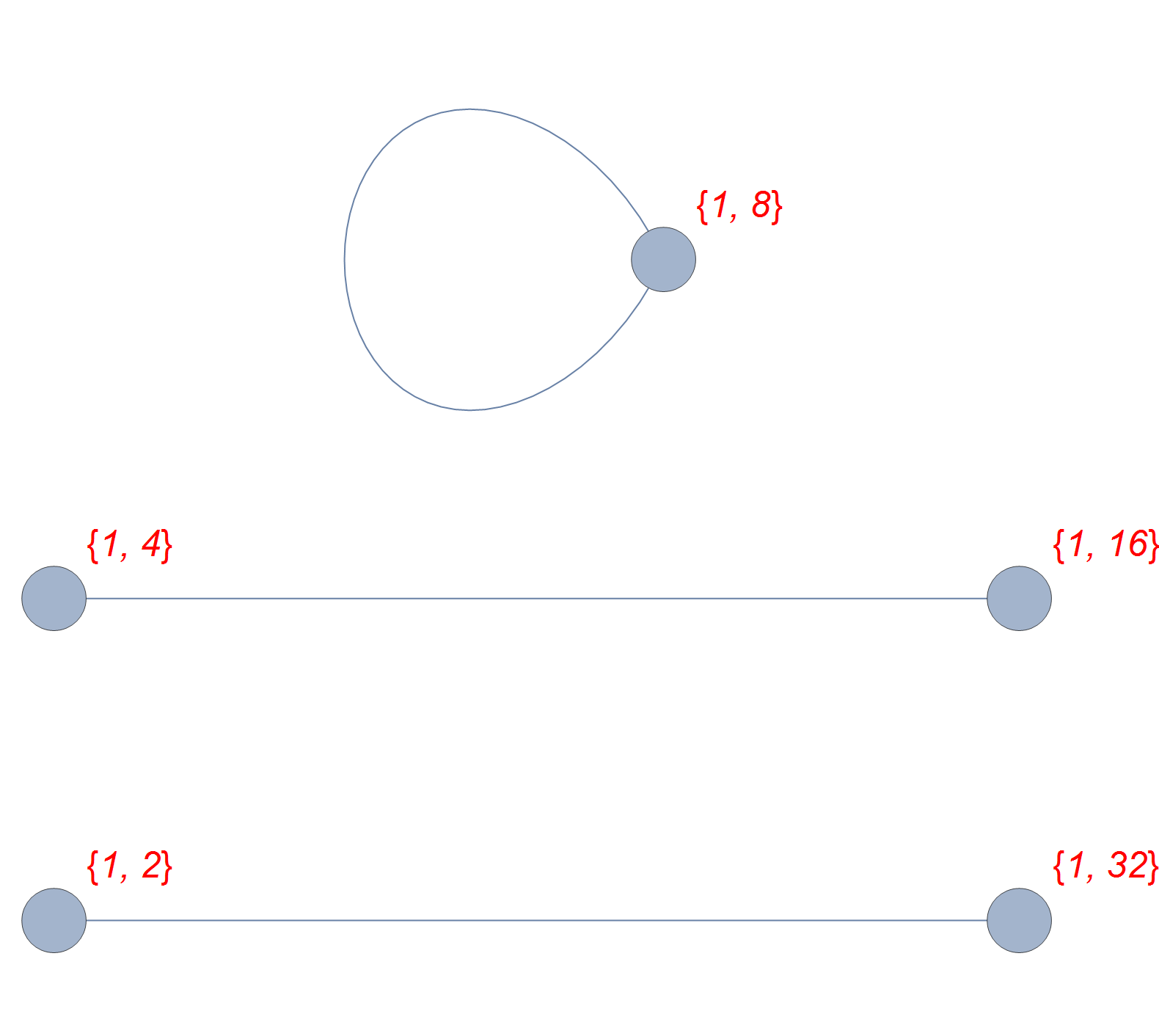}
\caption{G(32,2)}\end{figure}

In $G(16,2)$, we connect sets of the form $\set{1,2^a}$ where the powers on the 2 add up to 5. In $G(32,2)$, we connect sets of the form $\set{1,2^a}$ where the powers add up to 6.

We also have the following complete characterization of $G(2p,2)$, where $p$ is a prime other than 2:

\begin{thm}\label{pin2pthm}
    Let $p$ be a prime greater than $2$. Let $J,K\subseteq\{0,1,\ldots,2p -1\}$, $\abs{J}=\abs{K}=2$. Then $H_{J,K}$ is a Hadamard submatrix of $\mathcal{F}_{2p}$ if and only if ($\mathcal P_{2p} (J) , \mathcal P_{2p} (K)) $ is contained within $\{ (\{ 1 , 2 \} , \{ 1,2 \}) , (\{ 1 , 2 \} , \{ 1, 2p \}), (\{ 1, 2p \}, \{ 1 , 2 \} ) \} $. 
\end{thm}

\begin{proof}
    Suppose $H_{J,K}$ is Hadamard. Because $H_{J,K}$ has mutually orthogonal rows, for any $d\in\mathcal{D}(J)\setminus\set{0}$, $$ \sum_{k \in K} e^{\frac{2\pi i d k}{2 p}} = 0.$$ 
    Without loss of generality, we may assume that $j_1 = k_1 = 0$, so that $\mathcal{D}(J)=\{0,j_2,-j_2\}$ and $\mathcal{D}(K)=\{0,k_2,-k_2\}$. Therefore, taking $d=j_2$, $$0 = 1 +  e^{\frac{2\pi i j_2 k_2}{2p}}.$$ 
Therefore, $\ds e^{\frac{2\pi i j_2 k_2}{2p}} = -1$. 
This occurs when $\frac{j_2 k_2}{p} \in \mathbb Z_{\text{odd}}$.
Thus, $j_2 k_2 = y p , y \in \mathbb Z_{\text{odd}}$.
Therefore, $j_2$ and $k_2$ are odd, and at least one of them contains $p$ as a factor.

    If $p$ is a factor of $j_2$, then
		\begin{align*}\frac{2p}{\gcd(2p,d)}&=\frac{2p}{\gcd(2p,j_2)}\\
    &=\frac{2p}{p}\\
    &=2\in\mathcal{P}_{2p}(J).\end{align*}
	Similarly, if $p$ is a factor of $k_2$, then $2\in\mathcal{P}_{2p}(K)$.
    
If $p$ is not a factor of $j_2$, then 
\begin{align*}\frac{2p}{\gcd(2p,d)}&=\frac{2p}{\gcd(2p,j_2)}\\
    &=\frac{2p}{1}\\
    &=2p\in\mathcal{P}_{2p}(J).\end{align*}
Similarly, if $p$ is not a factor of $k_2$, then $2p\in\mathcal{P}_{2p}(K)$.
Taking $d=-j_2$ and $d=-k_2$ produces results that match with what was just shown. Finally, taking $d=0$ shows that $1\in\mathcal{P}(J)$ and $1\in\mathcal{P}(K)$ as usual.
Therefore,  ($\mathcal P_{2p} (J) , \mathcal P_{2p} (K)) $ is contained within $\{ (\{ 1 , 2 \} , \{ 1,2 \}) , (\{ 1 , 2p \} , \{ 1, 2 \}), (\{ 1 , 2 \} , \{ 1, 2p \}) \} $.
    
    Conversely, suppose  $(\mathcal P_{2p} (J) , \mathcal P_{2p} (K))\in\set{ (\{ 1 , 2 \} , \{ 1,2 \}) , (\{ 1 , 2p \} , \{ 1, 2 \}), (\{ 1 , 2 \} , \{ 1, 2p \})}$.
    Without loss of generality, we may assume $J=\{0,j_2\}$ and $K=\{0,k_2\}$. 
    Note that $\mathcal{P}_{2p}(J)=\set{1,2}$ or $\mathcal{P}_{2p}(K)=\set{1,2}$. By Theorem \ref{equivtheorem}, we may assume without loss of generality that $\mathcal{P}_{2p}(J)=\set{1,2}$. Thus we must have $2 =\frac{2p}{\gcd(2p,j_2)}$, implying that $\gcd(2p,j_2)=p$. Thus, $j_2=p y_j$, where $y_j$ is odd. 
		
		Since $\mathcal{P}_{2p}(K)=\set{1,2}$ or $\mathcal{P}_{2p}(K)=\set{1,2p}$, we must have $\frac{2p}{\gcd(2p,k_2)}=2$ or $\frac{2p}{\gcd(2p,k_2)}=2p$, implying that $2$ is not a factor of $\gcd(2p,k_2)$, and so $k_2$ is odd.
		
	Then
    \begin{align*}e^{\frac{2\pi i(j_2-j_1)k_1}{2p}}+e^{\frac{2\pi i(j_2-j_1)k_2}{2p}}&=1+e^{\frac{2\pi i j_2k_2}{2p}}\\
    &=1+e^{\frac{2\pi i p y_jk_2}{2p}}\\
    &=1+e^{\pi i y_jk_2}\\
    &=1-1=0.\end{align*}
    Hence, the two rows of $H_{J,K}$ are orthogonal, implying that $H_{J,K}$ is Hadamard.
\end{proof}

\begin{cor}Let $p$ be an odd prime. Then $V(G(2p,2))=\set{\set{1,2},\set{1,2p}}$ and $E(G(2p,2))=\set{(\set{1,2},\set{1,2}),(\set{1,2},\set{1,2p})}$.
\end{cor}

\begin{proof}The proof is immediate from Theorem \ref{pin2pthm}\end{proof}

Thus, $G(2p,2)$ is always a 2-vertex graph with $\set{1,2}$ and $\set{1,2p}$ as the vertices, an edge between them, and a loop from $\set{1,2}$ to itself.

The situation becomes a bit more complicated when $m$ has more prime factors. $G(m,2)$ can be completely constructed via the following result:

\begin{thm}\label{gen2by2thm}Suppose $J,K\subseteq\set{0,1,\ldots,m-1}$ with $\abs{J}=\abs{K}=2$. Then $H_{J,K}$ is a Hadamard submatrix of $\mathcal{F}_m$ if and only if $\nu_2^{min}(\mathcal{P}_m(J)\setminus\set{1})+\nu_2^{min}(\mathcal{P}_m(K)\setminus\set{1})=\nu_2^{max}(\mathcal{P}_m(J)\setminus\set{1})+\nu_2^{max}(\mathcal{P}_m(K)\setminus\set{1})=\nu_2(m)+1$ and $\nu_p^{max}(\mathcal{P}_m(J)\setminus\set{1})+\nu_p^{max}(\mathcal{P}_m(K)\setminus\set{1})\leq\nu_p(m)$ for all odd primes $p$.
   
\end{thm}

\begin{proof}
    First suppose $H_{J,K} $ is a Hadamard submatrix of $\mathcal F_m$. Without loss of generality, assume $J = \{ 0 , j_2 \}$ and $K = \{ 0 , k_2 \}$. 
   
    $H_{J,K}$ is Hadamard if and only if for all $d \in \mathcal D(J) \setminus \{0\}$, $\sum_{k\in K}{e^{\frac{2\pi i dk}{m}}} = 0.$ Since $\mathcal D(J) = \{ 0, \pm j_2 \}$, $H_{J,K}$ is Hadamard if and only if $0 = 1 + e^{\frac{2\pi i j_2k_2}{m}}$ and $0 = 1 + e^{-\frac{2\pi i j_2k_2}{m}}$. Hence, $H_{J,K}$ is Hadamard if and only if $\frac{2j_2k_2}{m} $ is an odd integer. 
   
    It follows that $H_{J,K}$ is Hadamard if and only if the following conditions hold:
    \begin{equation}\label{cond4}
        \begin{aligned}
            \nu_2(j_2)+\nu_2(k_2)+1       & =\nu_2(m),                                    \\
            \nu_{p}(j_2)+\nu_{p}(k_2) & \geq \nu_{p}(m) \text{ for all primes } p\neq2.
        \end{aligned}
    \end{equation}
   
    We will show these conditions are equivalent to
    \begin{equation}\label{cond5}\begin{aligned}
		&\nu_2^{min}(\mathcal{P}_m(J)\setminus\set{1})+\nu_2^{min}(\mathcal{P}_m(K)\setminus\set{1})\\
		&=\nu_2^{max}(\mathcal{P}_m(J)\setminus\set{1})+\nu_2^{max}(\mathcal{P}_m(K)\setminus\set{1})\\
		&=\nu_2(m)+1,\\
		&\text{and for all primes }p\neq2,\\
    &\nu_p^{max}(\mathcal{P}_m(J)\setminus\set{1})+\nu_p^{max}(\mathcal{P}_m(K)\setminus\set{1})\leq\nu_p(m),
		\end{aligned}\end{equation}
    which will complete the proof.
		
		By Proposition \ref{compprop}, for any prime $p$ we have
\begin{align*}&\nu_p^\text{max}(\mathcal{P}_m(J)\setminus\set{1})+\nu_p^\text{max}(\mathcal{P}_m(K)\setminus\set{1})\\
&\leq\text{max}\set{0,\nu_p(m)-\nu_p^\text{min}(\mathcal{D}(J)\setminus\set{0})}+\text{max}\set{0,\nu_p(m)-\nu_p^\text{min}(\mathcal{D}(K)\setminus\set{0})}.
\end{align*}
There are two possibilities for each maxima, and hence four possibilities total. Observe that:
$$0+0\leq\nu_p(m),$$
$$\nu_p(m)-\nu_p^\text{min}(\mathcal{D}(J)\setminus\set{0})+0\leq\nu_p(m),$$ and
$$0+\nu_p(m)-\nu_p^\text{min}(\mathcal{D}(K)\setminus\set{0})\leq\nu_p(m).$$
As for the fourth possibility, if $p=2$ we have
\begin{align*}&\nu_2(m)-\nu_2^\text{min}(\mathcal{D}(J)\setminus\set{0})+\nu_2(m)-\nu_2^\text{min}(\mathcal{D}(K)\setminus\set{0})\\
&=2\nu_2(m)-(\nu_2(m)-1)\\
&=\nu_2(m)+1,
\end{align*}
and if $p\neq 2$ we have
\begin{align*}&\nu_p(m)-\nu_p^\text{min}(\mathcal{D}(J)\setminus\set{0})+\nu_p(m)-\nu_p^\text{min}(\mathcal{D}(K)\setminus\set{0})\\
&=2\nu_p(m)-\nu_p^\text{min}(\mathcal{D}(J)\setminus\set{0})-\nu_p^\text{min}(\mathcal{D}(K)\setminus\set{0})\\
&\leq2\nu_p(m)-\nu_p(m)\\
&=\nu_p(m).
\end{align*}
Therefore,
$$\nu_2^\text{max}(\mathcal{P}_m(J)\setminus\set{1})+\nu_2^\text{max}(\mathcal{P}_m(K)\setminus\set{1})\leq\nu_2(m)+1$$ and for $p\neq 2$,
$$\nu_p^\text{max}(\mathcal{P}_m(J)\setminus\set{1})+\nu_p^\text{max}(\mathcal{P}_m(K)\setminus\set{1})\leq\nu_p(m).$$
By Proposition \ref{compprop}, we also have
\begin{align*}&\nu_2^\text{min}(\mathcal{P}_m(J)\setminus\set{1})+\nu_2^\text{min}(\mathcal{P}_m(K)\setminus\set{1})\\
&\geq2\nu_2(m)-\nu_2^\text{max}(\mathcal{D}(J)\setminus\set{0})-\nu_2^\text{max}(\mathcal{D}(K)\setminus\set{0})\\
&=2\nu_2(m)-(\nu_2(m)-1)\\
&=\nu_2(m)+1.
\end{align*}
Thus, $$\nu_2^\text{min}(\mathcal{P}_m(J)\setminus\set{1})+\nu_2^\text{min}(\mathcal{P}_m(K)\setminus\set{1})=\nu_2^\text{max}(\mathcal{P}_m(J)\setminus\set{1})+\nu_2^\text{max}(\mathcal{P}_m(K)\setminus\set{1})=\nu_2(m)+1.$$
This proves the forward implication.

Now suppose that $(\ref{cond5})$ holds. Since $\nu_2^\text{min}(\mathcal{P}_m(J)\setminus\set{1})+\nu_2^\text{min}(\mathcal{P}_m(K)\setminus\set{1})=\nu_2(m)+1$, $\nu_2^\text{min}(\mathcal{P}_m(J)\setminus\set{1})\leq\nu_2(m)$, and $\nu_2^\text{min}(\mathcal{P}_m(K)\setminus\set{1})\leq\nu_2(m)$, it follows that $\nu_2^\text{min}(\mathcal{P}_m(J)\setminus\set{1})\geq1$ and $\nu_2^\text{min}(\mathcal{P}_m(J)\setminus\set{1})\geq1$. So by Proposition \ref{compprop}, we have

\begin{align*}&\nu_2^\text{max}(\mathcal{D}(J)\setminus\set{0})+\nu_2^\text{max}(\mathcal{D}(K)\setminus\set{0})\\
&\leq2\nu_2(m)-\nu_2^\text{min}(\mathcal{P}_m(J)\setminus\set{1})-\nu_2^\text{min}(\mathcal{P}_m(K)\setminus\set{1})\\
&=2\nu_2(m)-(\nu_2(m)+1)\\
&=\nu_2(m)-1.
\end{align*}
Also by Proposition \ref{compprop},

\begin{align*}&\nu_2^\text{min}(\mathcal{D}(J)\setminus\set{0})+\nu_2^\text{min}(\mathcal{D}(K)\setminus\set{0})\\
&\geq2\nu_2(m)-\nu_2^\text{max}(\mathcal{P}_m(J)\setminus\set{1})-\nu_2^\text{max}(\mathcal{P}_m(J)\setminus\set{1})\\
&=\nu_2(m)-(\nu_2(m)+1)\\
&=\nu_2(m)-1.\end{align*}

Hence, $$\nu_2^\text{max}(\mathcal{D}(J)\setminus\set{0})+\nu_2^\text{max}(\mathcal{D}(K)\setminus\set{0})=\nu_2^\text{min}(\mathcal{D}(J)\setminus\set{0})+\nu_2^\text{min}(\mathcal{D}(K)\setminus\set{0})=\nu_2(m)-1.$$

Finally, for a prime $p\neq2$, by Proposition \ref{compprop} we have

\begin{align*}&\nu_p^\text{min}(\mathcal{D}(J)\setminus\set{0})+\nu_p^\text{min}(\mathcal{D}(K)\setminus\set{0})\\
&\geq2\nu_p(m)-\nu_p^\text{max}(\mathcal{P}_m(J)\setminus\set{1})-\nu_p^\text{max}(\mathcal{P}_m(J)\setminus\set{1})\\
&\geq\nu_p(m).\end{align*}

Since, $j_2\in\mathcal{D}(J)\setminus\set{0}$ and $k_2\in\mathcal{D}(K)\setminus\set{0}$, the conditions in $(\ref{cond4})$ follow immediately.
\end{proof}

Again, in the $n=2$ case, it is clear that any primitive set consists of two elements: 1 and another divisor of $m$. (Since there is only one element other than 1, of course $\nu_p^{\text{max}}(\mathcal{P}_m(J)\setminus\set{1})=\nu_p^{\text{min}}(\mathcal{P}_m(J)\setminus\set{1})$ in the $n=2$ case, but we retain the distinction in the statement of Theorem \ref{gen2by2thm} in order to parallel Theorem \ref{3by3thm}, where there is a difference.) Conversely, any set of the form $\set{1,d}$, where $d\neq1$ is a divisor of $m$, is realizable as the primitive set of a 2-element set, namely $\set{1,d}=\mathcal{P}_m(\set{0,m/d})$. Therefore, Theorem \ref{gen2by2thm} allows the complete construction of every graph $G(m,2)$: We simply take as vertices sets of the form $\set{1,d}$, where $d\neq1$ is a divisor of $m$, and connect them if the 2-adic orders of the $d$'s add up to one more than the 2-adic order of $m$ and the $p$-adic orders of the $d$'s add up to no more than the $p$-adic order of $m$ for any odd prime $p$.

For example, we can see this in the graph $G(180,2)$;

\begin{figure}[H]
\includegraphics[scale=.25]{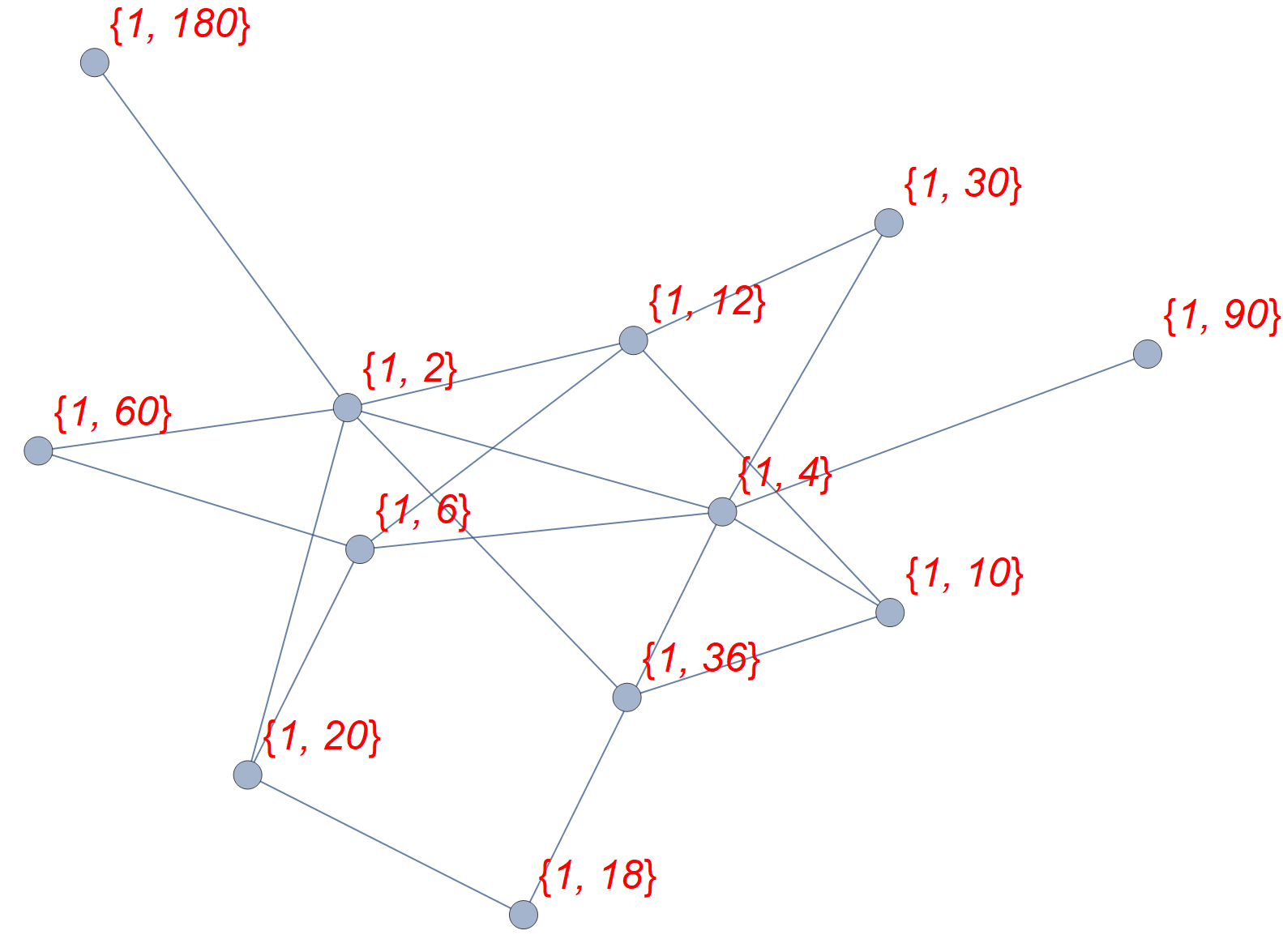}
\caption{G(180,2)}
\end{figure}

Since $180=2^2\cdot3^2\cdot5$, we connect sets of the form $\set{1,d}$, where $d\neq1$ is a divisor of 180, and such that the 2-adic orders of the $d$'s add up to exactly $\nu_2(180)+1=3$, the $3$-adic orders add up to no more than $\nu_3(180)=2$, and the 5-adic orders add up to no more than $\nu_5(180)=1$. For instance, $\set{1,20}$ and $\set{1,18}$ are connected because $\nu_2(20)+\nu_2(18)=2+1=3$, $\nu_3(20)+\nu_3(18)=0+2\leq2$, and $\nu_5(20)+\nu_5(18)=1+0\leq1$. On the other hand, $\set{1,18}$ and $\set{1,6}$ are not connected, because $\nu_3(18)+\nu_3(6)=2+1=3>2$.

It ought to be pointed out that the search for $2\times 2$ Hadamard submatrices boils down to an identification of the $-1$ entries in $\mathcal{F}_m$. Theorem $\ref{gen2by2thm}$ simply reframes the computation of these entries in terms of the primitive sets.

Via a similar but slightly different proof, we have the following result that eases construction of $G(m,3)$:

\begin{thm}\label{3by3thm}Suppose $J,K\subseteq\set{0,1,\ldots,m-1}$ with $\abs{J}=\abs{K}=3$. Then $H_{J,K}$ is a Hadamard submatrix of $\mathcal{F}_m$ if and only if \begin{align*}&\nu_3^\text{min}(\mathcal{P}_m(J)\setminus\set{1})+\nu_3^\text{min}(\mathcal{P}_m(K)\setminus\set{1})\\
&=\nu_3^\text{max}(\mathcal{P}_m(J)\setminus\set{1})+\nu_3^\text{max}(\mathcal{P}_m(K)\setminus\set{1})\\
&=\nu_3(m)+1\end{align*} and $$\nu_p^\text{max}(\mathcal{P}_m(J)\setminus\set{1})+\nu_p^\text{max}(\mathcal{P}_m(K)\setminus\set{1})\leq\nu_p(m)$$
 for all primes $p$ other than 3.
\end{thm}

\begin{proof}Without loss of generality, let $J=\set{0,j_2,j_3}$ and $K=\set{0,k_2,k_3}$. Observe that 
\begin{align*}\mathcal{D}(J)&=\set{0,\pm j_2,\pm j_3, \pm(j_3-j_2)}\\
\mathcal{D}(K)&=\set{0,\pm k_2,\pm k_3,\pm (k_3-k_2)}.\end{align*}

Suppose $H_{J,K}$ is Hadamard. As seen in \cite{TZ06}, which references \cite{Ha97}, all $3\times3$ Hadamard matrices are equivalent to $\mathcal{F}_3$. Since $H_{J,K}$ is dephased, there are only two possibilities for $H_{J,K}$:
$$H_{J,K}=\begin{bmatrix}1 & 1 & 1 \\ 1 & \zeta & \zeta^2 \\ 1 & \zeta^2 & \zeta\end{bmatrix}\text{ or }H_{J,K}=\begin{bmatrix}1 & 1 & 1 \\ 1 & \zeta^2 & \zeta \\ 1 & \zeta & \zeta^2\end{bmatrix},$$
where $\zeta=e^{2\pi i/3}$. 
Let $d_j\in\mathcal{D}(J)\setminus\set{0}$ and $d_k\in\mathcal{D}(K)\setminus\set{0}$. Then by the two possibilities for $H_{J,K}$, one of the following must be true:
\begin{itemize}\item For $v,w\in\set{2,3}$, $e^{ 2\pi\frac{d_jd_k}{m}}=e^{\pm2\pi i\frac{j_vk_w}{m}}\in\set{\zeta,\zeta^2}$
\item For $v\in\set{2,3}$, $e^{2\pi i\frac{d_jd_k}{m}}=e^{\pm 2\pi i\frac{j_v(k_3-k_2)}{m}}=e^{\pm 2\pi i\frac{j_vk_3}{m}}e^{\mp 2\pi i\frac{j_vk_2}{m}}\in\set{\zeta,\zeta^2}$
\item For $v\in\set{2,3}$, $e^{2\pi i\frac{d_jd_k}{m}}=e^{\pm 2\pi i\frac{k_v(j_3-j_2)}{m}}=e^{\pm 2\pi i\frac{j_3k_v}{m}}e^{\mp 2\pi i\frac{j_2k_v}{m}}\in\set{\zeta,\zeta^2}$
\item $e^{2\pi i \frac{d_jd_k}{m}}=e^{\pm 2\pi i\frac{(j_3-j_2)(k_3-k_2)}{m}}=e^{\pm 2\pi i\frac{(j_3k_3-j_3k_2-j_2k_3+j_2k_2)}{m}}=e^{\pm 2\pi i \frac{j_3(k_3-k_2)}{m}}e^{\pm 2\pi i \frac{j_2(k_2-k_3)}{m}}.$ Note that if $e^{\pm 2\pi i \frac{j_3(k_3-k_2)}{m}}=\zeta$, then $e^{\pm 2\pi i \frac{j_2(k_2-k_3)}{m}}=\zeta$. Also, if $e^{\pm 2\pi i \frac{j_3(k_3-k_2)}{m}}=\zeta^2$, then $e^{\pm 2\pi i \frac{j_2(k_2-k_3)}{m}}=\zeta^2$. Therefore, $e^{2\pi i\frac{d_jd_k}{m}}\in\set{\zeta,\zeta^2}$. 
\end{itemize}
Thus in all cases, $e^{2\pi i\frac{d_jd_k}{m}}$ is a primitive 3rd root of unity. It follows that $$\nu_3(d_j)+\nu_3(d_k)=\nu_3(m)-1$$ and for any prime $p$ other than $3$, $$\nu_p(d_j)+\nu_p(d_k)\geq\nu_p(m).$$
Since $d_j$ and $d_k$ were arbitrary, this shows that
\begin{equation}\label{3by3forward}
\begin{aligned}
\nu_3^\text{min}(\mathcal{D}(J)\setminus\set{0})+\nu_3^\text{min}(\mathcal{D}(K)\setminus\set{0})&=\nu_3(m)-1,\\
\nu_3^\text{max}(\mathcal{D}(J)\setminus\set{0})+\nu_3^\text{max}(\mathcal{D}(K)\setminus\set{0})&=\nu_3(m)-1\text{, and}\\
\nu_p^\text{min}(\mathcal{D}(J)\setminus\set{0})+\nu_p^\text{min}(\mathcal{D}(K)\setminus\set{0})&\geq\nu_p(m).\end{aligned}
\end{equation}
By Proposition \ref{compprop}, for any prime $p$ we have
\begin{align*}&\nu_p^\text{max}(\mathcal{P}_m(J)\setminus\set{1})+\nu_p^\text{max}(\mathcal{P}_m(K)\setminus\set{1})\\
&\leq\text{max}\set{0,\nu_p(m)-\nu_p^\text{min}(\mathcal{D}(J)\setminus\set{0})}+\text{max}\set{0,\nu_p(m)-\nu_p^\text{min}(\mathcal{D}(K)\setminus\set{0})}.
\end{align*}
There are two possibilities for each maxima, and hence four possibilities total. Observe that:
$$0+0\leq\nu_p(m),$$
$$\nu_p(m)-\nu_p^\text{min}(\mathcal{D}(J)\setminus\set{0})+0\leq\nu_p(m),$$ and
$$0+\nu_p(m)-\nu_p^\text{min}(\mathcal{D}(K)\setminus\set{0})\leq\nu_p(m).$$
As for the fourth possibility, by (\ref{3by3forward}), if $p=3$ we have
\begin{align*}&\nu_3(m)-\nu_3^\text{min}(\mathcal{D}(J)\setminus\set{0})+\nu_3(m)-\nu_3^\text{min}(\mathcal{D}(K)\setminus\set{0})\\
&=2\nu_3(m)-(\nu_3(m)-1)\\
&=\nu_3(m)+1,
\end{align*}
and if $p\neq 3$ we have
\begin{align*}&\nu_p(m)-\nu_p^\text{min}(\mathcal{D}(J)\setminus\set{0})+\nu_p(m)-\nu_p^\text{min}(\mathcal{D}(K)\setminus\set{0})\\
&=2\nu_p(m)-\nu_p^\text{min}(\mathcal{D}(J)\setminus\set{0})-\nu_p^\text{min}(\mathcal{D}(K)\setminus\set{0})\\
&\leq2\nu_p(m)-\nu_p(m)\\
&=\nu_p(m).
\end{align*}
Therefore,
$$\nu_3^\text{max}(\mathcal{P}_m(J)\setminus\set{1})+\nu_3^\text{max}(\mathcal{P}_m(K)\setminus\set{1})\leq\nu_3(m)+1$$ and for $p\neq 3$,
$$\nu_p^\text{max}(\mathcal{P}_m(J)\setminus\set{1})+\nu_p^\text{max}(\mathcal{P}_m(K)\setminus\set{1})\leq\nu_p(m).$$
By Proposition \ref{compprop} and (\ref{3by3forward}), we also have
\begin{align*}&\nu_3^\text{min}(\mathcal{P}_m(J)\setminus\set{1})+\nu_3^\text{min}(\mathcal{P}_m(K)\setminus\set{1})\\
&\geq2\nu_3(m)-\nu_3^\text{max}(\mathcal{D}(J)\setminus\set{0})-\nu_3^\text{max}(\mathcal{D}(K)\setminus\set{0})\\
&=2\nu_3(m)-(\nu_3(m)-1)\\
&=\nu_3(m)+1.
\end{align*}
Thus, $$\nu_3^\text{min}(\mathcal{P}_m(J)\setminus\set{1})+\nu_3^\text{min}(\mathcal{P}_m(K)\setminus\set{1})=\nu_3^\text{max}(\mathcal{P}_m(J)\setminus\set{1})+\nu_3^\text{max}(\mathcal{P}_m(K)\setminus\set{1})=\nu_3(m)+1.$$
This proves the forward implication. 

Suppose conversely that
\begin{equation}\label{3x3backward1}
\begin{aligned}&\nu_3^\text{min}(\mathcal{P}_m(J)\setminus\set{1})+\nu_3^\text{min}(\mathcal{P}_m(K)\setminus\set{1})\\
&=\nu_3^\text{max}(\mathcal{P}_m(J)\setminus\set{1})+\nu_3^\text{max}(\mathcal{P}_m(K)\setminus\set{1})\\
&=\nu_3(m)+1
\end{aligned}\end{equation}
 and 
\begin{align}\label{3x3backward2}\nu_p^\text{max}(\mathcal{P}_m(J)\setminus\set{1})+\nu_p^\text{max}(\mathcal{P}_m(K)\setminus\set{1})\leq\nu_p(m)
\end{align}
for all primes $p$ other than 3.

For a prime $p$ other than $3$, by $(\ref{3x3backward2})$ and Proposition \ref{compprop}, we have
\begin{align*}&\nu_p^\text{min}(\mathcal{D}(J)\setminus\set{0})+\nu_p^\text{min}(\mathcal{D}(K)\setminus\set{0})\\
&\geq2\nu_p(m)-\nu_p^\text{max}(\mathcal{P}_m(J)\setminus\set{1})-\nu_p^\text{max}(\mathcal{P}_m(K)\setminus\set{1})\\
&\geq2\nu_p(m)-\nu_p(m)\\
&=\nu_p(m).
\end{align*}
For $p=3$, (\ref{3x3backward1}) and Proposition \ref{compprop} imply
\begin{align*}&\nu_3^\text{min}(\mathcal{D}(J)\setminus\set{0})+\nu_3^\text{min}(\mathcal{D}(K)\setminus\set{0})\\
&\geq2\nu_3(m)-\nu_3^\text{max}(\mathcal{P}_m(J)\setminus\set{1})-\nu_3^\text{max}(\mathcal{P}_m(K)\setminus\set{1})\\
&=\nu_3(m)-1.
\end{align*}

Since $\mathcal{P}_m(J)$ and $\mathcal{P}_m(K)$ contain only factors of $m$, it follows that $\nu_3^\text{min}(\mathcal{P}_m(J)\setminus\set{1})\leq\nu_3(m)$ and $\nu_3^\text{min}(\mathcal{P}_m(K)\setminus\set{1})\leq\nu_3(m)$. Since by (\ref{3x3backward1}) $$\nu_3^\text{min}(\mathcal{P}_m(J)\setminus\set{1})+\nu_3^\text{min}(\mathcal{P}_m(K)\setminus\set{1})=\nu_3(m)+1,$$ it follows that $\nu_3^\text{min}(\mathcal{P}_m(J)\setminus\set{1})\geq1$ and $\nu_3^\text{min}(\mathcal{P}_m(K)\setminus\set{1})\geq1$. Therefore, we may apply Proposition \ref{compprop} and obtain
\begin{align*}&\nu_3^\text{max}(\mathcal{D}(J)\setminus\set{0})+\nu_3^\text{max}(\mathcal{D}(K)\setminus\set{0})\\
&\leq2\nu_3(m)-\nu_3^\text{min}(\mathcal{P}_m(J)\setminus\set{1})-\nu_3^\text{min}(\mathcal{P}_m(K)\setminus\set{1})\\
&=2\nu_3(m)-(\nu_3(m)+1)\\
&=\nu_3(m)-1.\end{align*}
Thus,
\begin{align*}&\nu_3^\text{max}(\mathcal{D}(J)\setminus\set{0})+\nu_3^\text{max}(\mathcal{D}(K)\setminus\set{0})\\
&\nu_3^\text{min}(\mathcal{D}(J)\setminus\set{0})+\nu_3^\text{min}(\mathcal{D}(K)\setminus\set{0})\\
&=\nu_3(m)-1.\end{align*}
It follows that for any $d_j\in\mathcal{D}(J)\setminus\set{0}$ and $d_k\in\mathcal{D}(K)\setminus\set{0}$, $e^{2\pi i\frac{d_jd_k}{m}}$ is a primitive 3rd root of unity, i.e. $\zeta$ or $\zeta^2$. Since $j_2,j_3\in\mathcal{D}(J)\setminus\set{0}$ and $k_2,k_3\in\mathcal{D}(K)\setminus\set{0}$, it follows that
$$H_{J,K}=\begin{bmatrix}1 & 1 & 1\\
1 & a & b \\
1 & c & d\end{bmatrix},$$
where $a,b,c,d\in\set{\zeta,\zeta^2}$. We also have that
\begin{align*}\frac{a}{b}&=e^{2\pi i\frac{j_2k_2}{m}}e^{-2\pi i\frac{j_2k_3}{m}}\\
&=e^{2\pi i\frac{j_2(k_2-k_3)}{m}}\\
&\in\set{\zeta,\zeta^2},\end{align*}
because $j_2\in\mathcal{D}(J)\setminus\set{0}$ and $(k_2-k_3)\in\mathcal{D}(K)\setminus\set{0}$. It follows that $a\neq b$. By similar reasoning, $a\neq c$ and $b\neq d$. Hence,

$$H_{J,K}=\begin{bmatrix}1 & 1 & 1 \\ 1 & \zeta & \zeta^2 \\ 1 & \zeta^2 & \zeta\end{bmatrix}\text{ or }H_{J,K}=\begin{bmatrix}1 & 1 & 1 \\ 1 & \zeta^2 & \zeta \\ 1 & \zeta & \zeta^2\end{bmatrix},$$
and so $H_{J,K}$ is Hadamard.
\end{proof}

Thus, for $G(m,3)$ we get a similar description of when an edge appears between vertices based on comparing the $p$-adic orders of the elements to the $p$-adic orders of $m$. It is important to note that the theorem assumes one is testing primitive sets arising from $3$-element sets. In the $n=3$ case, the primitive sets can have either 2 or 3 elements, but not all $2$- or $3$-element sets of divisors of $m$ are realizable as $m$th primitive sets (see, for example, Theorem \ref{mutexcthm}).

As an example of Theorem \ref{3by3thm}, consider the graph $G(180,3)$, shown below:

\begin{figure}[H]
\includegraphics[scale=.25]{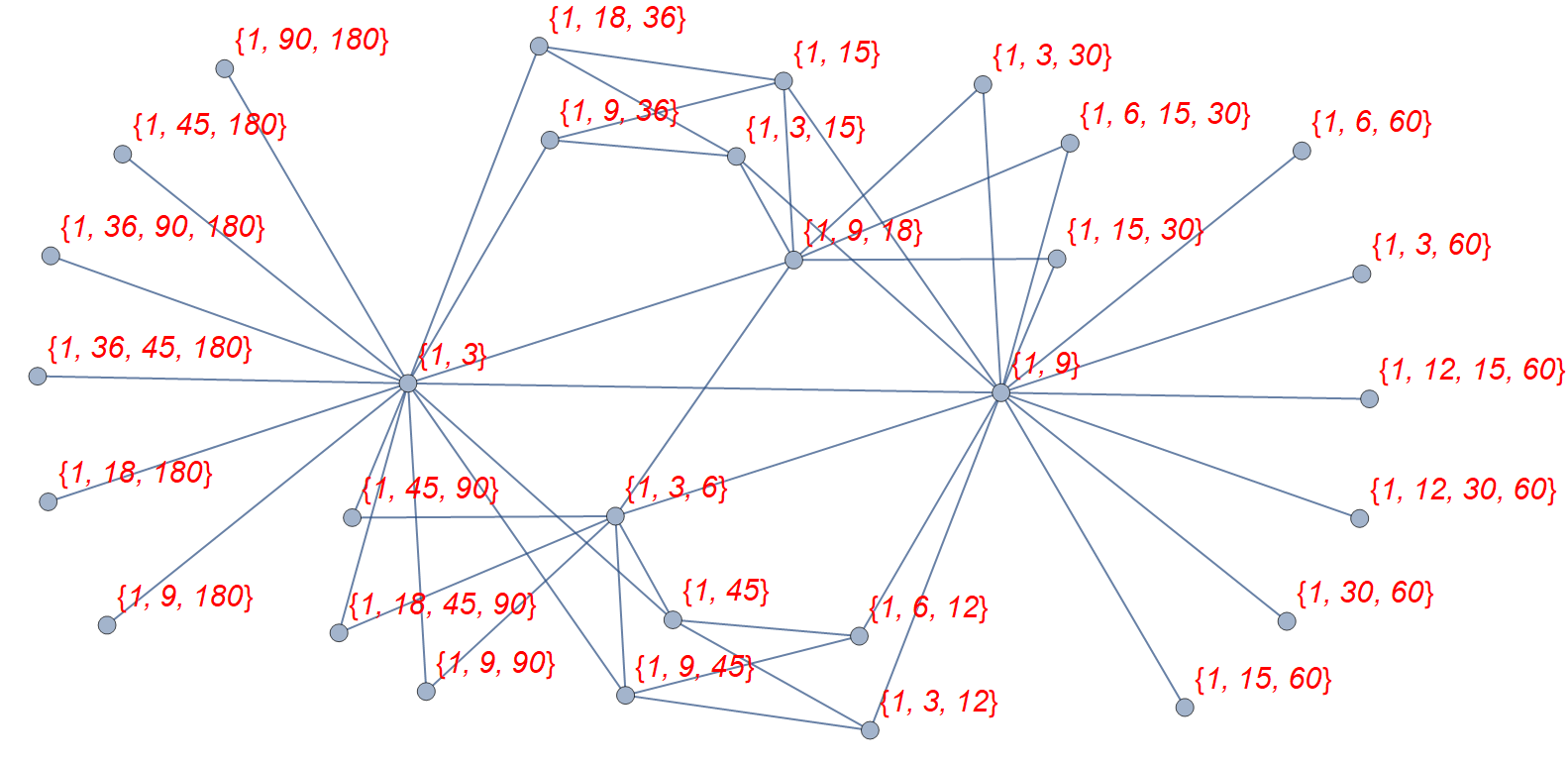}
\caption{G(180,3)}\label{R180n3}
\end{figure}

Since $180=2^2\cdot3^2\cdot5$, Theorem \ref{3by3thm} implies there will be an edge between valid primitive sets if and only if the maximum and minimum $3$-adic orders add up to exactly 3, the maximum $2$-adic orders add up to no more than 2, and the maximum 5-adic orders add up to no more than 1.

Let us pick, say, $\set{1,9,45}$ and $\set{1,6,12}$. We have:
\begin{align*} \nu_3^\text{max}(\set{9,45})+\nu_3^\text{max}(\set{6,12})&=2+1=3;\\
\nu_3^\text{min}(\set{9,45})+\nu_3^\text{min}(\set{6,12})&=2+1=3;\\
\nu_2^\text{max}(\set{9,45})+\nu_2^\text{max}(\set{6,12})&=0+2\leq2;\\
\nu_5^\text{max}(\set{9,45})+\nu_5^\text{max}(\set{6,12})&=1+0\leq1.
\end{align*}

Thus, by Theorem \ref{3by3thm}, $\set{1,9,45}$ and $\set{1,6,12}$ should be connected in the graph $G(180,3)$, and we see in Figure \ref{R180n3} that they are. On the other hand, we notice that $\set{1,15,60}$ and $\set{1,30,60}$ are not connected. This is because $\nu_2^\text{max}(\set{1,15,60})+\nu_2^\text{max}(\set{1,30,60})=2+2=4>2.$

\section{Other Compatibility Graph Examples}

Using the computer program Mathematica, the authors have conducted full searches for Hadamard submatrices of certain sizes for Fourier matrices of certain sizes, and have thus been able to fully construct $G(m,n)$ for a number of combinations of $m$ and $n$. Some of these graphs were shown in the previous section.

These exhaustive searches, however, are much cruder than what the results in this paper make possible. The efficiency afforded by Theorem \ref{equivtheorem} is that once one has tested a particular $J$ and $K$ with primitive sets $\mathcal{P}_m(J)$ and $\mathcal{P}_m(K)$, any other $J$ and $K$ with the same primitive sets will similarly form a Hadamard matrix or not, and so no computer check of them is necessary. Hence in theory, the compatibility graph of a Fourier matrix can be constructed much faster than an exhaustive search for its Hadamard submatrices. Further efficiencies are made possible by the relationship of compatibility graphs to each other via such results as Theorem \ref{mutexcthm} and Theorem \ref{PrimThm}.

To satiate the reader's curiosity, we conclude by displaying a few compatibility graphs outside the $G(m,2)$ and $G(m,3)$ cases:

\begin{figure}[H]
\includegraphics[scale=.25]{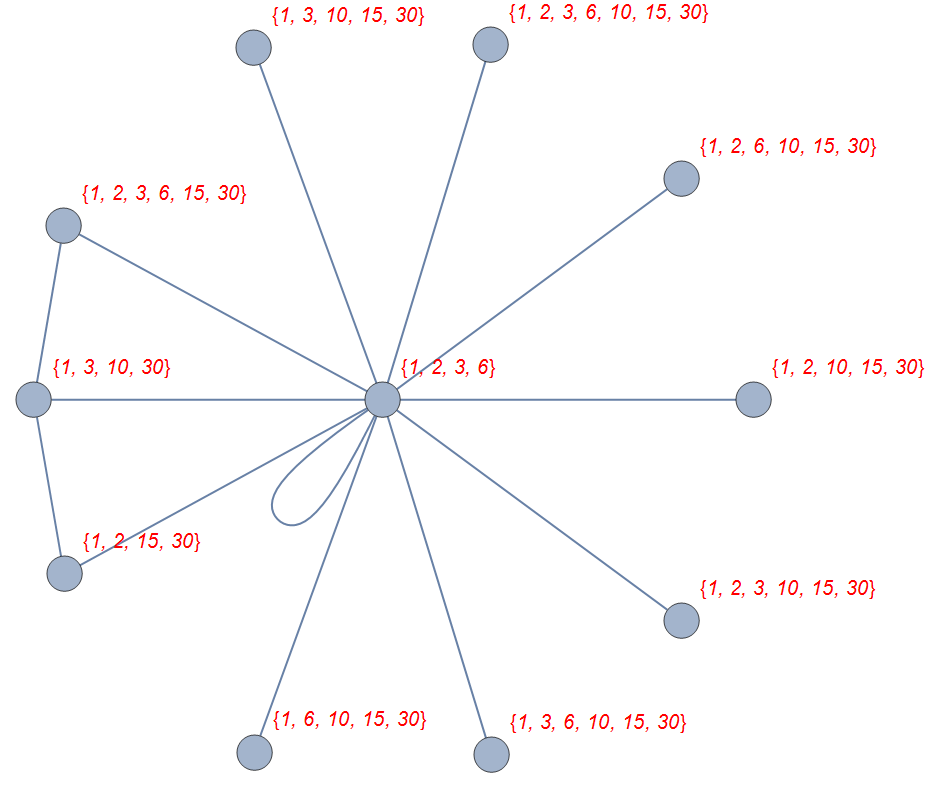}
\caption{G(30,6)}
\end{figure}

Note that in the case of $G(30,6)$ above we have a dominant primitive set $\set{1,2,3,6}$, in the sense that it is compatible with all primitive sets that appear in the graph. This behavior is also on display in the graph $G(36,4)$ below, and many others:

\begin{figure}[H]
\includegraphics[scale=.25]{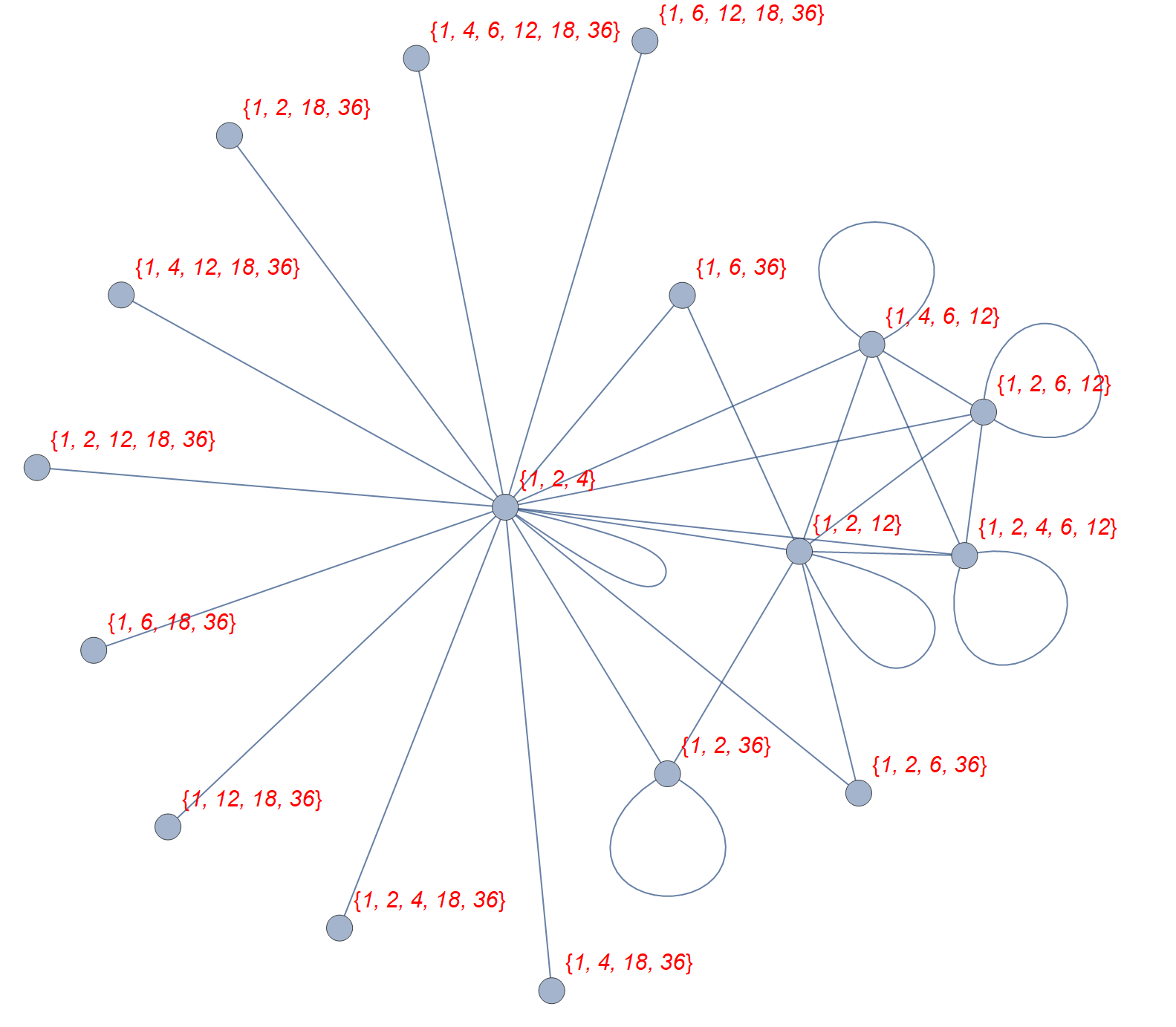}
\caption{G(36,4)}
\end{figure}

\section{Conclusions}

In this paper, we have shown that primitive set compatibility determines whether row and column selections in a Fourier matrix yield a Hadamard submatrix. This allows us to group submatrices into equivalence classes defined by the combination of primitive sets coming from their row and column sets, each equivalence class consisting entirely of Hadamard or entirely of non-Hadamard submatrices. The equivalence classes that do consist of Hadamard submatrices may be represented as edges in ``compatibility graphs,'' which visually represent the Hadamard submatrix structure of a Fourier matrix. We have also shown that primitive set compatibility in one submatrix size precludes compatibility in any other size, regardless of the order of the Fourier matrix.

Since our motivation comes from the application of Hadamard triples to the Universal Tiling Conjecture, for future directions we would like to explore the relationship between the primitive set of a row selection and its tiling set.

\bibliographystyle{amsalpha}
\bibliography{HerrWiegand2020bib}
\nocite{*}

\end{document}